\documentclass[11pt,a4paper]{article}
\usepackage[english]{babel}
\usepackage{amsmath,amsfonts,amsthm,amssymb,amsbsy,upref,color,graphicx,amscd,hyperref,makeidx,enumerate,bbm,comment}
\usepackage{mathtools}
\usepackage[a4paper,margin=2.5truecm]{geometry}
\usepackage[active]{srcltx}
\usepackage[latin1]{inputenc}

\DeclareMathOperator{\dive}{div}

\DeclareMathOperator{\spt}{spt}

\DeclareMathOperator{\diam}{diam}

\def\ds{\displaystyle}
\def\eps{{\varepsilon}}

\def\R{\mathbb{R}}

\def\A{\mathcal{A}}

\def\HH{\mathcal{H}}

\def\M{\mathcal{M}}

\def\Dr{D}

\newcommand{\be}{\begin{equation}}
\newcommand{\ee}{\end{equation}}
\newcommand{\bib}[4]{\bibitem{#1}{\sc#2: }{\it#3. }{#4.}}

\numberwithin{equation}{section}
\theoremstyle{plain}
\newtheorem{teo}{Theorem}[section]
\newtheorem{lemma}[teo]{Lemma}

\newtheorem{prop}[teo]{Proposition}
\newtheorem{deff}[teo]{Definition}

\theoremstyle{remark}
\newtheorem{oss}[teo]{Remark}
\newtheorem{exam}[teo]{Example}

\newenvironment{ack}{{\bf Acknowledgements. }}

\def\widthh{5}
\def\widthhh{8}

\title{A free boundary problem arising in PDE optimization}

\author{Giuseppe Buttazzo, Edouard Oudet, Bozhidar Velichkov}

\date{May 26, 2015}

\begin{document}

\maketitle

\begin{abstract}
A free boundary problem arising from the optimal reinforcement of a membrane or from the reduction of traffic congestion is considered; it is of the form
$$\sup_{\int_D\theta\,dx=m}\ \inf_{u\in H^1_0(D)}\int_D\Big(\frac{1+\theta}{2}|\nabla u|^2-fu\Big)\,dx.$$
We prove the existence of an optimal reinforcement $\theta$ and that it has some higher integrability properties. We also provide some numerical computations for $\theta$ and $u$. 

\end{abstract}



\section{Introduction}\label{sintro}

In the present paper we consider a variational problem related to various models arising in several fields among which shape optimization, optimal transport and elasto-plasticity. Here below we discuss the statement of the problem in some of them, the original motivation for the paper being the first one, for which we refer to \cite{budm93}, \cite{bubu05} and references therein.

\medskip
{\it Optimal reinforcement of a membrane.} Let $\Dr\subset\R^d$ be a bounded open set with a smooth boundary and let $f\in L^2(\Dr)$; $\Dr$ represents (when $d=2$) a membrane, that we assume fixed at its boundary, and $f$ represents a given load. The deformation $u$ of the membrane is then obtained as the solution of the variational problem
$$\min\Big\{\int_\Dr\Big(\frac12|\nabla u|^2-fu\Big)\,dx\ :\ u\in H^1_0(\Dr)\Big\},$$
where the elasticity coefficient in $D$ is taken equal to $1$, or equivalently of the PDE
$$-\Delta u=f\quad\hbox{in}\quad\Dr,\qquad u=0\quad\hbox{on}\quad\partial\Dr.$$
We now want to rigidify the membrane by adding a reinforcement: this is modeled by a density function $\theta\ge0$ which corresponds to changing the elasticity coefficient of $\Dr$ from $1$ to $1+\theta$. The total amount $m$ of reinforcement is given and the ultimate goal is to increase as much as possible the rigidity of the membrane. This translates into the minimization of the elastic compliance, or equivalently into the maximization of the elastic energy, and we end up with the following optimization problem:
\be\label{membrane}
\sup_{\int_\Dr\theta\,dx=m}\ \inf_{u\in H^1_0(\Dr)}\int_\Dr\Big(\frac{1+\theta}{2}|\nabla u|^2-fu\Big)\,dx.
\ee

\medskip
{\it Reduction to a traffic congestion problem.} Consider $\Dr$ as an urban region with given traffic sources $f$; if $\sigma$ denotes the exiting traffic flux in $\Dr$ (or equivalently towards the boundary $\partial\Dr$) and if the cost of the traffic congestion in $\Dr$ is measured by the quantity $\int_D\frac12|\sigma|^2\,dx$ (see for instance \cite{cajisa}, \cite{bucagu}), we may obtain $\sigma$ by solving the minimization problem
$$\min\Big\{\int_\Dr\frac12|\sigma|^2\,dx\ :\ -\dive\sigma=f\hbox{ in }\Dr\Big\}.$$
Suppose that there is some given total amount $m$ of resources that can be spent in order to decrease the traffic congestion: we assume that investing the amount $\theta\,dx$ near a point $x_0$ produces a lower congestion, measured by the quantity $\frac{1}{2(1+\theta)}|\sigma|^2$. Therefore, minimizing the total congestion cost leads to the problem
$$\min\Big\{\int_\Dr\frac{1}{2}\frac{|\sigma|^2}{1+\theta}\,dx\ :\ -\dive\sigma=f\hbox{ in }\Dr\Big\},$$
whose dual form is well-known to coincide with the optimization problem \eqref{membrane}.

\medskip
The variational problem we deal with has been considered in the literature under various forms, below we mention two of them. 

\medskip
{\it The elastic-plastic torsion problem.} Let $\Dr\subset\R^2$ be the section of a thin rod which is being twisted and let $f$ be a constant corresponding to the twist of the rod. The function $u\in H^1_0(\Dr)$ is a potential related to the elastic energy of the twisted rod. The response of some materials in this situation is of elastic-plastic type which means that the gradient of the potential $u$ cannot grow beyond a certain threshold $\kappa$, i.e. $u$ solves the variational problem 
$$\min\Big\{\frac12\int_\Dr|\nabla u|^2\,dx-f\int_\Dr u\,dx\ :\ u\in H^1_0(\Dr),\ |\nabla u|\le\kappa\Big\},$$
the {\it elastic region} being given by $\{|\nabla u|<\kappa\}$ while the {\it plastic} one by $\{|\nabla u|=\kappa\}$.

In \cite{brezis} Brezis proved that the potential $u$ of the elastic-plastic torsion problem is in fact the solution of the problem
$$-\dive\big((1+\theta)\nabla u\big)=f\quad\text{in }\Dr,\qquad u=0\quad\text{on }\partial\Dr,$$
where $\theta$ is precisely the solution of the optimization problem \eqref{membrane} for some $m>0$. In fact $\theta$ in this case can be interpreted as the response of the material to the deformation of the rod. 

Several authors studied the regularity of the torsion function $u$ and the free boundary $\partial\{|\nabla u|<\kappa\}$. The optimal $C^{1,1}$ regularity of $u$ was proved by Wiegner \cite{wiegner} and Evans \cite{evansC11}, while Caffarelli and Friedman provided a detailed study of the free boundary in \cite{caffri}.

\medskip

{\it Optimal transport problem} In \cite{boubut} Bouchitt\'e and Buttazzo considered a problem similar to \eqref{membrane} without the underlying membrane, that is with $1+\theta$ replaced by $\theta$, and made the connection of the optimization problem \eqref{membrane} with the optimal transport theory. Several authors (see for instance \cite{depr}, \cite{deevpr}, \cite{fil}) studied the summability properties of the function $\theta$ (called {\it transport density}) in relation to the summability of the function $f$ (the positive and negative parts $f^+$ and $f^-$ are usually called {\it marginals}). The usual setting considered in the literature is with $D$ convex, Neumann conditions at $\partial\Dr$, together with the zero average condition $\int_\Dr f\,dx=0$. In our case the Dirichlet boundary condition requires additional regularity of $\partial\Dr$ to get summability properties of $\theta$. Moreover, the presence of the underlying membrane gives raise to a free boundary problem that we analyze in detail.

\medskip
The paper is organized as follows. The precise mathematical setting of the optimization problem and our main results are described in Section \ref{spres}, while the proof is carried out in Section \ref{proofsec04} . Section \ref{sobst} is devoted to the {\it elastic-plastic torsion problem}, in which $f=1$. Section \ref{snum} deals with some numerical computations providing the solutions $\bar u$ and $\bar\theta$ in some interesting situations.

\section{Setting of the problem and main result}\label{spres}

Let $\Dr\subset\R^d$ be a given bounded open set and $f\in L^2(\Dr)$.
For a given function $u\in H^1_0(\Dr)$ and $\theta \in L^1_+(\Dr):=\Big\{\theta\in L^1(\Dr)\ :\ \theta\ge 0\Big\}$, the functional $J_f:H^1_0(\Dr)\times L^1_+(\Dr)\to \R$ is defined as 
$$J_f(u,\theta)=\frac12\int_\Dr (1+\theta)|\nabla u|^2\,dx-\int_{\Dr}uf\,dx,$$
and the energy $E_f(\theta)$ is given by
\be\label{minu04}
E_f(\theta)=\min\Big\{J_f(u,\theta)\ :\ u\in H^1_0(\Dr)\Big\}.
\ee
The minimum in \eqref{minu04} is always achieved due to the fact that every minimizing sequence is bounded in $H^1_0(\Dr)$ and so weakly compact in $H^1_0(\Dr)$. The minimizer is also unique due to the strict convexity of the functional $J_f$ with respect to the first variable. If we denote it by $u_\theta$ we have that $u_\theta\in H^1_0(\Dr)$, $|\nabla u_\theta|\in L^2(\theta\,dx)$ and $u_\theta$ is the weak solution of 
\be\label{utheta04}
-\dive\big((1+\theta)\nabla u_\theta\big)=f\quad\text{in}\quad\Dr,\qquad u_\theta\in H^1_0(\Dr),
\ee
i.e. for every $\varphi\in H^1_0(\Dr)$ such that $|\nabla\varphi| \in L^2(\theta dx)$ we have 
$$\int_{\Dr}(1+\theta)\nabla u_\theta\cdot\nabla \varphi\,dx=\int_{\Dr}f\varphi\,dx.$$
Testing the equation \eqref{utheta04} with $\varphi=u_\theta$ and integrating by parts leads to
$$E_f(\theta)=J_f(u_\theta,\theta)=-\frac12\int_\Dr u_\theta f\,dx.$$

We consider the optimization problem 
\be\label{optbp04}
\max_{\theta\in\A_m} E_f(\theta),
\ee
where, for a given $m>0$, the admissible set $\A_m$ is given by
$$\A_m=\Big\{\theta\in L^1_+(\Dr)\ :\ \int_\Dr\theta(x)\,dx\le m\Big\}.$$
Due to the monotonicity of $E_f(\theta)$ with respect to $\theta$ the formulation of the problem with the alternative admissible set
$$\widetilde{\A}_m=\Big\{\theta\in L^1_+(\Dr)\ :\ \int_\Dr\theta(x)\,dx=m\Big\},$$
is equivalent to the one with $\A_m$.

\begin{oss}\label{fmeasure}
It is interesting to notice that the optimization problem \eqref{optbp04} is meaningful even if the load $f$ is assumed to be a measure; in this case we have to allow $\theta$ to be a measure too, and for every nonnegative measure $\theta$ the functional $J_f(u,\theta)$ and the energy $E_f(\theta)$ are defined by
\[\begin{split}
&J_f(u,\theta)=\frac12\int_\Dr(1+\theta)|\nabla u|^2\,dx-\int_\Dr u\,df,\qquad\forall u\in C^\infty_c(\Dr)\\
&E_f(\theta)=\inf\Big\{J_f(u,\theta)\ :\ u\in C^\infty_c(\Dr)\Big\}.
\end{split}\]
If $f$ is a measure of course we may have $E_f(\theta)=-\infty$ for some measures $\theta$; this happens for instance when $f$ concentrates on sets of dimension smaller than $d-1$ and $\theta$ is the Lebesgue measure. However, these ``singular'' measures $\theta$ are ruled out from our analysis because of the maximization problem \eqref{optbp04} we are dealing with.
\end{oss}

\begin{oss}
Testing the energy $E_f(\theta)$ by the function $u=0$ gives immediately the inequality
$$E_f(\theta)\le0\qquad\forall\theta\in\A_m.$$
A more careful analysis (see for instance \cite{boubut}) shows the better inequality
$$E_f(\theta)\le-\frac{I^2(f)}{2(|\Dr|+m)}\qquad\forall\theta\in\A_m$$
where
$$I(f)=\sup\Big\{\int_\Dr u\,df\ :\ u\in C^\infty_c(\Dr),\ |\nabla u|\le1\Big\}.$$
On the other hand, if $f\in H^{-1}(\Dr)$ the energy $E_f$ remains always bounded from below.:
\[\begin{split}
J_f(u,\theta)&\ge\frac12\int_\Dr(1+\theta)|\nabla u|^2\,dx-\|f\|_{H^{-1}(\Dr)}\|u\|_{H^1_0(\Dr)}\\
&\ge\frac12\int_\Dr|\nabla u|^2\,dx-\frac12\big(\|f\|^2_{H^{-1}(\Dr)}+\|u\|^2_{H^1_0(\Dr)}\big)\ge-\frac12\|f\|^2_{H^{-1}(\Dr)}.
\end{split}\]
Taking the minimum over $u\in H^1_0(\Dr)$ we obtain
$$E_f(\theta)\ge-\frac12\|f\|^2_{H^{-1}(\Dr)}\qquad\forall\theta\in\A_m.$$
When $f\in L^\infty(\Dr)$, setting $M=\|f\|_\infty$ and using the inequalities $E_f(\theta)\ge E_f(0)\ge E_M(0)$, which come from the maximum principle, we also have the inequality
$$E_f(\theta)\ge-\frac{\|f\|^2_{L^\infty}}{2}\int_{\Dr}w_\Dr\,dx,$$
where $w_\Dr$ is the solution of 
$$-\Delta w_\Dr=1\quad\hbox{in}\quad\Dr,\qquad w_\Dr\in H^1_0(\Dr).$$
\end{oss}

Our main result is the following.

\begin{teo}\label{mainth04}
Let $\Dr\subset\R^d$ be a bounded open set satisfying the external ball condition and let $f\in L^\infty(\Dr)$. Then there is a solution to the problem \eqref{optbp04}. Moreover, any solution $\theta$ of \eqref{optbp04} has the following properties:
\begin{enumerate}[(i)]
\item \emph{(higher integrability)} $\theta\in L^p(\Dr)$, for every $p\ge1$;
\item \emph{(min-max exchange)} the energy $E_f(\theta)$ satisfies
\be\label{optcond04}
E_f(\theta)=\min_{u\in H^1_0(\Dr)}\ \frac12\int_{\Dr}|\nabla u|^2\,dx+\frac{m}2\|\nabla u\|_{L^\infty}^2-\int_\Dr uf\,dx;
\ee
\item \emph{(regularity of the state function)} the solution $u_\theta$ of \eqref{utheta04} is also the minimizer of the right-hand side of \eqref{optcond04}. In particular, if $f\in C^2(\Dr)$ then $u_\theta\in C^{1,1}(\Dr)$;
\item \emph{(optimality condition)} $\theta=0$ a.e. on the set $\{|\nabla u_\theta|< \|\nabla u_\theta\|_{L^\infty}\}$.
\end{enumerate}
\end{teo}

In the case when the set $\Dr$ is convex we have a slightly stronger result. 

\begin{teo}\label{mainth04conv}
Let $\Dr\subset\R^d$ be a bounded convex set, $r\in(d,+\infty]$ and $f\in L^r(\Dr)$. Then all the conclusions of Theorem \ref{mainth04} hold. Moreover for the maximizer $\theta$ we have that $\theta\in L^r(\Dr)$. 
\end{teo}

Before we pass to the proof of Theorem \ref{mainth04} and Theorem \ref{mainth04conv} we give some preliminary remarks in order to clarify the assumptions on $\Dr$ and $f$ as well as to outline the main steps of the proof.

The {\it main difficulty} in the existence result Theorem \ref{mainth04} is the lack of compactness in the space of admissible functions $\A_m$. Our approach consists in considering the optimization problem \eqref{optbp04} on the admissible class
\be\label{amp}
\A_{m,p}=\Big\{\theta\in L^p(\Dr),\ \theta\ge0,\ \Big(\int_\Dr\theta^p\,dx\Big)^{1/p}\le m\Big\},
\ee
and then to pass to the limit as $p\to1$. In fact, due to the weak compactness of the unit ball in $L^p$ (for $p>1$) we can easily obtain the existence of an optimal reinforcement $\theta_p$ for
$$\max\big\{E_f(\theta)\ :\ \theta\in\A_{m,p}\big\}.$$
We discuss this problem in the beginning of Section \ref{proofsec04}. The rest of the proof is dedicated to the passage to the limit as $p\to1$. We notice that at this point we work with not just a maximising sequence of $L^1$ functions but with a family go maxima that obey some regularity properties we may use in order to get the desired convergence result. 

Our goal is to find a \emph{uniform estimate} for $\theta_p$ in some Lebesgue space $L^r(\Dr)$ with $r>1$. In order to do that we use an estimate from \cite{deevpr} in a slightly more general form. Roughly speaking, it is of the form 
$$\int_\Dr\theta_p^r\,dx\le C_r\int_\Dr|f(x)|^r\,dx-\int_{\partial\Dr}F_r(|\nabla u_{\theta_p}|)H_{\partial\Dr}\,d\HH^{d-1},$$
where $C_r$ is a constant depending on $r$, $F_r$ is a polynomial with positive coefficients depending only on $r$ and $H_{\partial\Dr}$ is the mean curvature of $\partial\Dr$. In the case when $\Dr$ is convex the boundary integral is positive, which immediately provides the uniform estimate which finally leads to the result of Theorem \ref{mainth04conv}. On the other hand, if $\Dr$ is not convex then we should estimate the boundary term in a uniform way. At this point the bound from below of the mean curvature is natural to be imposed, but we have to estimate also the gradient of the solution on the boundary. We do this by constructing a barrier for $u_{\theta_p}$ at each point of the boundary. In order to do this we use the external ball condition, which is stronger than just the bound on the mean curvature, and the assumption $f\in L^\infty$ which allows us to describe the barrier as an explicit solution in a suitably chosen domain. The construction of the barrier and the uniform estimate is proved in Section \ref{proofsec04}.

In the case when $f$ {\it is a constant} the proof of Theorem \ref{mainth04} can be obtained in a more direct way. In fact once the bound on the gradient is established on $\partial\Dr$ one can extend it in the interior of $\Dr$ by using the Payne-Philippin maximum principle \cite{paph}. In this case also the boundedness of the optimal $\theta$ is achieved. We notice that this result was already proved by Brezis in \cite{brezis} where the optimal $\theta$ is given through an explicit formula. 

Once we prove that the family $\theta_p$ is uniformly bounded in $L^r$, we take, as a natural candidate for a solution, the weak limit of $\theta_p$ as $p\to1$. In order to prove that the limit is a minimizer of \eqref{optbp04}, we use the characterisation of the torsion function $u_p$ as minima of functionals not depending only on $f$ and $p$. 

There is \emph{an alternative approach} to problem \eqref{optbp04}. In fact, the existence of an optimal reinforcement can be obtained for very general $\Dr$ and $f$ if we relax the problem to the space of finite Borel measures on the compact set $\overline\Dr$. Moreover, for an optimal measure $\theta$ the min-max exchange equality \eqref{optcond04} still holds. Now under the assumption that $\Dr$ is convex one can use the regularity result for $u_\theta$ proved in \cite{bresta} and then writing the equation for $\theta$ as an optimal transport problem in which $\theta$ corresponds to the density of transport rays one can obtain the absolute continuity and the summability of $\theta$ using the results from \cite{deevpr} and \cite{fil}. We discuss this approach in the separate Section \ref{alt04}.

\begin{oss}
It is interesting to analyze the behavior of the solution $u_m$ of \eqref{optcond04} as $m\to+\infty$. Suppose that $f\ge 0$ and set $v_m=mu_m$. We have that $v_m$ solves the minimization problem
$$\min_{v\in H^1_0(\Dr)}\ \frac1{2m}\int_{\Dr}|\nabla v|^2\,dx+\frac12\|\nabla v\|_{L^\infty(\Dr)}^2-\int_\Dr vf\,dx\;.$$
The $\Gamma$-limit, as $m\to+\infty$ of the cost functional in the last line is
$$\frac12\|\nabla v\|_{L^\infty(\Dr)}^2-\int_\Dr vf\,dx\;,$$
hence $v_m=mu_m$ converges in $L^2(\Dr)$ to the solution $\bar v$ of the minimization problem
$$\min_{v\in H^1_0(\Dr)}\ \frac12\|\nabla v\|_{L^\infty(\Dr)}^2-\int_\Dr vf\,dx\;,$$
which is given by $\bar v(x)=C_\Dr d_{\partial\Dr}(x)$, where $d_{\partial\Dr}(x)$ is the Euclidean distance from $x\in\Dr$ to the boundary $\partial\Dr$ and $C_\Dr=\int_\Dr d_{\partial\Dr}(x)f(x)\,dx$.
\end{oss}

\section{Existence of optimal reinforcement}\label{proofsec04}
In this section we carry out the proofs of Theorem \ref{mainth04} and Theorem \ref{mainth04conv}. In order to make the presentation as clear as possible we divide the main steps of the proof into different subsections.
 
\subsection{Approximating problems: existence and regularity for $p>1$}
In this section we consider the problem of finding an optimal reinforcement in the admissible class of reinforcements $\A_{m,p}$ given by \eqref{amp}. Precisely, for every $p>1$ and $f\in L^2(\Dr)$, we consider the optimization problem
\be\label{optp}
\max\big\{E_f(\theta)\ :\ \theta\in\A_{m,p}\big\}=\max_{\theta\in\A_{m,p}}\ \min_{u\in H^1_0(\Dr)}J_f(u,\theta).
\ee
Interchanging the $\sup$ and the $\inf$ in \eqref{optp} we obtain the inequality
\be\label{exchinfsup}
\sup_{\theta\in\A_{m,p}}\ \min_{u\in H^1_0(\Dr)}\ J_f(u,\theta)\le \inf_{u\in H^1_0(\Dr)}\ \sup_{\theta\in\A_{m,p}}J_f(u,\theta).
\ee
We now notice that in the H\"older inequality $\int \theta \varphi\,dx\le \|\theta\|_{L^p}\|\varphi\|_{L^q}$ the equality is achieved if and only if $\theta=\|\theta\|_{L^p}\|\varphi\|_{L^q}^{1-q}\varphi^{q-1}$. Thus, for a fixed $u\in H^1_0(\Omega)$ the supremum on the right-hand side of \eqref{exchinfsup} with respect $\theta\in\A_{m,p}$ can be computed explicitly:
$$\sup_{\theta\in\A_{m,p}}J_f(u,\theta)=\frac12\int_\Dr|\nabla u|^2\,dx+\frac{m}{2}\Big(\int_\Dr|\nabla u|^{2q}\,dx\Big)^{1/q}-\int_\Dr fu\,dx,$$
where $q=p/(p-1)$ is the dual exponent of $p>1$. If moreover $u\in W^{1,2q}_0(\Dr)$, then the supremum is in fact a maximum and is achieved for
$$\theta=m|\nabla u|^{2(q-1)}\Big(\int_\Dr|\nabla u|^{2q}\,dx\Big)^{-(q-1)/q}.$$
Thus, for a given $u$, we can recover in a unique way the candidate for optimal $\theta$ in \eqref{optp} and thus it is now sufficient to optimize in $u$. We make this argument precise in the following proposition.

\begin{prop}\label{exisp}
For every $p>1$ and $f\in L^2(\Dr)$ the optimization problem \eqref{optp} admits a unique solution $\theta_p$, given by
\be\label{thetapthunder}
\theta_p=m |\nabla u_p|^{2(q-1)}\Big(\int_\Dr|\nabla u_p|^{2q}\,dx\Big)^{-(q-1)/q},
\ee
where $q=p/(p-1)$ is the dual exponent of $p$ and $u_p$ is the solution of the auxiliary problem
\be\label{auxpbp}
\min\Big\{\frac12\int_\Dr|\nabla u|^2\,dx+\frac{m}{2}\Big(\int_\Dr|\nabla u|^{2q}\,dx\Big)^{1/q}-\int_\Dr fu\,dx\ :\ u\in H^1_0(\Dr)\Big\}.
\ee
Moreover, if $\Dr\subset\R^d$ is a bounded domain with $C^{1,\alpha}$ boundary and $f\in L^\infty(\Dr)$, then $u_p\in C^{1,\beta}(\bar \Dr)$ for a constant $\beta$ depending on $\alpha$, $p$, $d$, $\Omega$ and $\|f\|_{L^\infty}$. In particular, $\theta_p$ is H\"older continuous up to the boundary.
\end{prop}

\begin{proof}
As we already noticed, by exchanging the max and the min in the original problem \eqref{optp} we get 
\be\label{supminminsupp}
\sup_{\theta\in\A_{m,p}}\ \min_{u\in H^1_0(\Dr)}\ J_f(u,\theta)\le \inf_{u\in H^1_0(\Dr)}\Big\{\frac12\int_\Dr|\nabla u|^2\,dx+\frac{m}{2}\Big(\int_\Dr|\nabla u|^{2q}\,dx\Big)^{1/q}-\int_\Dr fu\,dx\Big\}.
\ee
By the direct methods of the calculus of variations we have that the infimum in the right-hand side of \eqref{supminminsupp} is achieved, the minimizer $u_p$ is unique and satisfies (weakly in $W^{1,2q}_0(\Dr)$) the equation 
$$-\Delta u_p-mC_p\Delta_{2q} u_p=f\quad\hbox{in }\Dr,\qquad u_p\in W^{1,2q}_0(\Dr),$$
where the constant $C_p$ is given by $\ds C_p=\|\nabla u_p\|_{L^{2q}}^{-2(q-1)}$ and, for $r>1$, $\Delta_r$ is the $r$-Laplacian operator
$$\Delta_r u=\dive\big(|\nabla u|^{r-2}\nabla u\big).$$
Defining $\theta_p$ as in \eqref{thetapthunder}, we get that $u_p$ satisfies 
$$-\dive\big((1+\theta_p)\nabla u_p\big)=f\quad\hbox{in }\Dr,\qquad u_p\in H^1_0(\Dr),$$
in sense of distributions in $\Dr$ and so $u_p$ is the unique minimizer of $J_f(\cdot,\theta_p)$ in $H^1_0(\Omega)$ and 
$$E_f(\theta_p)=J_f(u_p,\theta_p)=\frac12\int_\Dr|\nabla u_p|^2\,dx+\frac{m}{2}\Big(\int_\Dr|\nabla u_p|^{2q}\,dx\Big)^{1/q}-\int_\Dr fu_p\,dx,$$
which together with \eqref{supminminsupp} and the minimality of $u_p$ gives
$$\sup_{\theta\in\A_{m,p}}\ E_f(\theta)\le E_f(\theta_p).$$
By the choice of $\theta_p$ we have $\int_{\Dr}\theta_p^p\,dx=m^p$ and by the fact that $u_p$ is the unique minimizer of \eqref{auxpbp}, we get that $\theta_p$ is the unique solution of \eqref{optp}. 

The internal $C^{1,\beta}$-regularity of $u_p$ (and, as a consequence, the H\"older-regularity of $\theta_p$) is classical and can be found in \cite{ladura}, while the up to the boundary version was proved in Theorem 1 of \cite{lieberman}.
\end{proof}

\subsection{The limit, as $p\to1$, of the state functions $u_p$}
Let $\Dr\subset\R^d$ be an open set of finite measure, $f\in L^2(\Dr)$ be a given function, $m>0$ and $p>1$ two fixed numbers; we denote by $q=p/(p-1)$ the dual exponent of $p>1$.

We define the functional $F_p:L^2(\Dr)\to\R$ as 
\be\label{Fp04}
F_p(u)=\begin{cases}
\ds\frac12\int_\Dr|\nabla u|^2\,dx+\frac{m}{2}\Big(\int_\Dr|\nabla u|^{2q}\,dx\Big)^{1/q}-\int_\Dr fu\,dx,&\ds\qquad\hbox{if }u\in W^{1,2q}_0(\Dr),\\
+\infty,&\ds\qquad\hbox{otherwise},
\end{cases}
\ee
while $F_1: L^2(\Dr)\to \R$ is given by 
\be\label{F104}
F_1(u)=\begin{cases}
\ds\frac12\int_\Dr|\nabla u|^2\,dx+\frac{m}{2}\|\nabla u\|_{L^\infty(\Dr)}^2-\int_\Dr fu\,dx,&\ds\qquad\hbox{if }u\in W^{1,\infty}_0(\Dr),\\
+\infty,&\ds\qquad\hbox{otherwise}.
\end{cases}
\ee
For $p\ge1$ we denote by $u_p$ the unique minimizer of $F_p$. 

Our main objective in this subsection is to show that $u_p$ converges to $u_1$ in some suitable functional space (a natural candidate being $L^2(\Dr)$) and that the sequence of norms $\|\nabla u_p\|_{L^{2q}}$ does not degenerate as $p\to 1$, a fact that we use later in the uniform estimate of $\|\theta_p\|_{L^\infty}$. Before we continue with the result from this section we recall the notion of $\Gamma$-convergence introduced by De Giorgi: 

\begin{deff}
Let $X$ be a given metric space. We say that a sequence of functionals $F_n: X\to \R\cup\{+\infty\}$ $\Gamma$-converges to the functional $F: X\to \R\cup\{+\infty\}$, if 
\begin{itemize}
\item {$(\Gamma-\limsup)$} For every sequence $x_n\in X$ converging to $x\in X$ we have $$\ds F(x)\le \liminf F_n(x_n);$$
\item {$(\Gamma-\liminf)$} For every $x\in X$ there is a sequence $x_n\in X$ converging to $x$ and such that $$\ds F(x)\ge \limsup F_n(x_n).$$
\end{itemize}
\end{deff}

\begin{prop}\label{convergenceprop}
Let $\Dr\subset\R^d$ be an open set of finite measure and let $f\in L^2(\Dr)$. With the notation introduced in \eqref{Fp04} and \eqref{F104} we have:
\begin{enumerate}[(i)]
\item The sequence of functionals $(F_p)_{p>1}$ $\Gamma$-converges in $L^2(\Dr)$ to the functional $F_1$.
\item The sequence of minima $(u_p)_{p>1}$ converges strongly in $H^1(\Dr)$ to $u_1$. 
\item $\ds \lim_{p\to 1}\|\nabla u_p\|_{L^{2q}}=\|\nabla u_1\|_{L^\infty}$.
\end{enumerate}
\end{prop}

\begin{proof}
The proof is well-known, let us give it for the sake of completeness. Suppose that
$$v_p\to v\quad\text{in}\quad L^2(\Dr)\qquad\text{and}\qquad\lim_{p\to 1} F_p(v_p)<+\infty.$$
Then there is a constant $C>0$ such that for $p>1$ small enough we have 
\be\label{lemmaXeq1}
\frac12\|\nabla v_p\|_{L^2}^2+\frac{m}2\|\nabla v_p\|_{L^{2q}}^2-\int_\Dr v_pf\,dx\le C.
\ee
Thus, by the Gagliardo-Nirenberg-Sobolev inequality we have that $v_p$ is uniformly bounded in $H^1_0(\Dr)$ and so we can suppose that $v_p$ converges weakly in $H^1_0(\Dr)$ to $v\in H^1_0(\Dr)$ (which gives the semicontinuity of the first term of \eqref{lemmaXeq1}). We now notice that 
$$\|\nabla v\|_{L^\infty(\Dr)}=\sup\Big\{\int_\Dr \nabla v\cdot \nu\,dx\ :\ \nu\in C^\infty_c(\Dr;\R^d),\ \int_\Dr |\nu|\,dx\le 1\Big\}.$$
On the other hand for a fixed $\nu\in C^\infty_c(\Dr;\R^d)$ with $\int_\Dr |\nu|\,dx\le 1$, we have 
\[\begin{split}
\int_\Dr\nabla v\cdot\nu\,dx
&=\lim_{p\to\infty}\int_\Dr\nabla v_p\cdot\nu\,dx\\
&\le\Big(\int_\Dr|\nabla v_p|^{2q}\,dx\Big)^{1/(2q)}\Big(\int_\Dr|\nu|^{2q/(2q-1)}\,dx\Big)^{(2q-1)/(2q)}\\
&\le\Big(\liminf_{p\to 1}\|\nabla v_p\|_{L^{2q}}\Big)\,\Big(\lim_{p\to1}\|\nu\|_{L^{2p/(p+1)}}\Big)\\
&\le\|\nu\|_{L^1}\,\liminf_{p\to1}\|\nabla v_p\|_{L^{2q}},
\end{split}\]
which proves the semicontinuity of the second term in \eqref{lemmaXeq1}.

In order or prove the $\Gamma-\limsup$ inequality consider a function $v\in W^{1,\infty}_0(\Dr)$ and take $v_p$ to be the constant family $v_p=v$, for every $p>1$. Since in this case we have $\|\nabla v\|_{L^{2q}}\to\|\nabla v\|_{L^\infty}$ as $q\to+\infty$, we obtain 
$$F(v)=\lim_{p\to1}F_p(v),$$
which concludes the proof of {\it (i)}. 

We now prove {\it (ii)}. Since $u_p$ is the minimizer of $F_p$ and $F_p(0)=0$, we have that $F_p(u_p)\le 0$, which gives 
$$\frac12\|\nabla u_p\|_{L^2(\Dr)}^2\le\int_\Dr u_p f\,dx\le\|u_p\|_{L^2}\|f\|_{L^2}\le \lambda_1(\Dr)^{-1/2}\|\nabla u_p\|_{L^2} \|f\|_{L^2},$$
which gives the uniform (in $p$) bound $\|\nabla u_p\|_{L^2}\le 2\|f\|_{L^2}\lambda_1(\Dr)^{-1/2}$. Thus, there is a subsequence of $u_p$ converging in $L^2$ to some $u\in L^2(\Dr)$. Since $F_p$ $\Gamma$-converges to $F_1$, $u$ is necessarily the minimizer $u_1$ of $F_1$ (and since for every subsequence the limit is the same $u_p$ converges to $u_1$ in $L^2$). Moreover, by the optimality of $u_p$ and the $\Gamma-\liminf$ inequality we have 
$$\limsup_{p\to1}F_p(u_p)\le \liminf_{p\to1}F_p(u_1)=F_1(u_1)\le \liminf_{p\to1}F_p(u_p),$$
which together with the fact that $\int_\Dr u_p f\,dx\to\int_\Dr uf\,dx$ gives that
$$\|\nabla u_1\|_{L^2}^2+m\|\nabla u_1\|_{L^\infty}^2=\lim_{p\to1}\Big\{\|\nabla u_p\|_{L^2}^2+m\|\nabla u_p\|_{L^{2q}}^2\Big\}.$$
Now since both the terms are semicontinuous with respect to the $L^2$ convergence of $u_p$ to $u_1$ we get that 
$$\|\nabla u_1\|_{L^2}^2=\lim_{p\to1}\|\nabla u_p\|_{L^2}^2\qquad\text{and}\qquad \|\nabla u_1\|_{L^\infty}^2=\lim_{p\to1}\|\nabla u_p\|_{L^{2q}}^2.$$
The first equality proves that $u_p\to u_1$ strongly in $H^1(\Dr)$, while the second equality is precisely {\it (iii)}.
\end{proof}

\subsection{Construction of a barrier for $u_p$}

In order to obtain uniform estimates for $u_p$ and $\theta_p$ we impose some mild geometric assumptions on the domain $\Dr$. Precisely we assume that {\it $\Dr$ satisfies a uniform external ball condition that is there is some $\rho>0$ such that for every point of the boundary $x_0\in\partial\Dr$ there is a ball $B_\rho(y_0)\subset\R^d\setminus\Dr$ such that $x_0\in\partial B_\rho(y_0)$}. This assumption allows us to construct a barrier for $u_p$ and thus to give an estimate on the gradient $|\nabla u_p|$ on the boundary $\partial\Dr$ in terms of the diameter of $\Dr$ and the largest possible $\rho$. In the case when $\Dr$ is convex we have an estimate only in terms of the diameter $\diam(\Dr)$. In this estimate we use a weak maximum (or comparison) principle which we state below in a general form. 

\begin{lemma}[Weak maximum principle]
Let $\Dr\subset\R^d$ be a bounded open set and let $G:[0,+\infty)\to[0,+\infty)$ be a convex function such that $G'(0)>0$. For a function $f\in L^2(\Dr)$ we denote by $u_{\Omega,f}$ the unique minimizer of the problem
$$\min\Big\{\frac12\int_{\Dr} G\big(|\nabla u|^2\big)\,dx-\int_\Dr uf\,dx\ :\ u\in H^1_0(\Dr)\Big\}.$$
\begin{enumerate}[(a)]
\item If $f,g\in L^2(\Dr)$ are such that $f\ge g$, then $u_{\Omega,f}\ge u_{\Omega,g}$.
\item If $\omega\subset\Omega$ is a bounded open set and $f\ge 0$ a.e. on $\Omega$, then $u_{\Omega,f}\ge u_{\omega,f}$. 
\end{enumerate}
\end{lemma}

We are now in a position to construct our barrier.

\begin{lemma}[Pointwise boundary estimate]
Let $\Dr\subset\R^d$ be a bounded open set whose boundary $\partial \Dr$ is locally a graph of a $C^{1,\alpha}$ function for some $\alpha\in(0,1)$.
Let $x_0\in \partial\Dr$ and let $B_\rho(y_0)$ be a ball in $\R^d\setminus \Dr$ tangent to $\partial\Dr$ in $x_0$. Then, the minimizer $u_p$ of $F_p$ satisfies the gradient estimate 
\be\label{estextbal0}
\Big(1+m C_p |\nabla u_p(x_0)|^{2(q-1)}\Big)|\nabla u_p(x_0)|\le \|f\|_{L^\infty(\Dr)}\left(1+\frac{\diam(\Dr)}{\rho}\right)^{d-1}\diam(\Dr).
\ee
\end{lemma}

\begin{proof}
Without loss of generality we assume $y_0=0$. Let $R$ be the diameter of $\Dr$. Then $\Dr\subset B_{R+\rho}\setminus\bar{B_\rho}$. We set for simplicity
$$M=\|f\|_{L^\infty}\;,\qquad C=mC_p=m\Big(\int_\Dr|\nabla u_p|^{2q}\,dx\Big)^{-1/p}$$
and we stress that $M$ and $C$ are fixed constants in the proof below. Let $u$ be the solution of the equation
$$-\dive\Big(\big(1+C|\nabla u|^{2(q-1)}\big)\nabla u\Big)=M\quad\hbox{in }\Dr,\qquad u\in W^{1,2q}_0(\Dr),$$
and let $U$ be the solution of 
\be\label{radialplap}
-\dive\Big(\big(1+C|\nabla U|^{2(q-1)}\big)\nabla U\Big)=M\quad\hbox{in }B_{R+\rho}\setminus\bar{B_\rho},\qquad U\in W^{1,2q}_0(B_{R+\rho}\setminus\bar{B_\rho}).
\ee
By the maximum principle, applied to the operator associated to the function $G(t)=t+Ct^{q-1}$, we have $|u_p|\le u\le U$. On the other hand, the solution $U$ of \eqref{radialplap} is radially symmetric, i.e. a function of $r$ that we still denote by $U$ and that solves
\be\label{radialplap1}
-r^{1-d}\partial_r\Big(r^{d-1}\big(1+C|\partial_r U|^{2(q-1)}\big)\partial_r U\Big)=M\quad\hbox{in }]\rho,\rho+R[,\qquad U(\rho)=U(\rho+R)=0.
\ee
Let $\rho_1\in ]\rho, \rho+R[$ be a point where $U$ achieves its maximum. Integrating \eqref{radialplap1} we have
$$\rho^{d-1}\big(1+C|\partial_r U(\rho)|^{2(q-1)}\big)\partial_r U(\rho)=\int_{\rho}^{\rho_1}M r^{d-1}\,dr\le M R(R+\rho)^{d-1},$$
which gives \eqref{estextbal0} since $\partial_r U(\rho)>0$.
\end{proof}

Taking the largest possible $\rho$ in the pointwise estimate \eqref{estextbal0} we obtain the following result. 

\begin{lemma}[Boundary estimate for a regular domain]\label{barrierlemma}
Let $\Dr\subset\R^d$ be a bounded open set whose boundary $\partial \Dr$ is locally a graph of a $C^{1,\alpha}$ function for some $\alpha\in(0,1)$.
\begin{enumerate}[(i)]
\item If $\Dr$ is convex, then we have 
$$\Big(1+m C_p\|\nabla u_p\|_{L^\infty(\partial\Dr)}^{2(q-1)}\Big)\|\nabla u_p\|_{L^\infty(\partial\Dr)}\le\diam(\Dr)\|f\|_{L^\infty(\Dr)}.$$
\item If $\Dr$ satisfies the external ball condition with radius $\rho$, then
$$\Big(1+m C_p\|\nabla u_p\|_{L^\infty(\partial\Dr)}^{2(q-1)}\Big)\|\nabla u_p\|_{L^\infty(\partial\Dr)}\le \|f\|_{L^\infty(\Dr)}\diam(\Dr)\left(1+\frac{\diam(\Dr)}{\rho}\right)^{d-1}.$$
\end{enumerate}
\end{lemma}

\subsection{Uniform estimate for $\theta_p$}
In this section we extend the boundary estimate for $\theta_p$ to the entire domain $\Dr$. We need a uniform estimate on the norm $\|u_p\|_{L^\infty}$, which is classical and so we give only the precise statement here.

\begin{lemma}\label{inftyestlem04}
Let $f\in L^r(\Dr)$ for some $r>d/2$. Then we have the estimate 
$$\|u_p\|_{L^\infty}\le \frac{C_d}{d/2-1/r}\|f\|_{L^r}|\Dr|^{d/2-1/r},$$
where $C_d$ is a dimensional constant.
\end{lemma}

We now give our main a priori estimate for $\theta_p$. The statement might appear as a generalization of an estimate by De Pascale-Evans-Pratelli, but the proof is precisely the one they gave in \cite{deevpr}. We reproduce the proof below for the sake of completeness and in order to show that the presence of a general functional $G$ and the non-convexity of the set $\Dr$ do not influence the final result. In fact the only difference is that since the domain is not convex there is a boundary term that appears in the final estimate.

\begin{lemma}[De Pascale-Evans-Pratelli]\label{apriorilemma}
Let $\Dr$ be a bounded open set with smooth boundary and let $f\in L^r(\Dr)$ for $r\ge2$ and $r>d/2$. Let $G:[0,+\infty)\to[0,+\infty)$ be a convex function such that $G'(0)>0$ and let $u\in C^{1}(\bar\Dr)\cap H^2_{loc}(\Dr)$ be the solution of 
\be\label{eqGaldogigievans}
-\dive\big(G'(|\nabla u|^2)\nabla u\big)=f\quad\hbox{in }\Dr,\qquad u=0\quad\hbox{on }\partial\Dr.
\ee
Then for every $\eps\in(0,1)$ we have the estimate
\[\begin{split}
\int_\Dr\big|G'(|\nabla u|^2)\big|^r|\nabla u|^2\,dx
&\le3\eps\int_\Dr\big|G'(|\nabla u|^2)\big|^r\,dx+\Big(\frac{(r-1)^r}{\eps^{r-1}}+\frac1{\eps^{2r-1}}\Big)\|u\|_{L^\infty}^r\|f\|_{L^r}^r\\
&\qquad\qquad-\frac{(r-1)^2\|u\|_{L^\infty}^2}\eps\int_{\partial\Dr}H\big|G'(|\nabla u|^2)\big|^r|\nabla u|^2\,d\HH^{d-1},
\end{split}\]
where $H$ is the mean curvature of $\partial\Dr$ with respect to the outer normal. 
\end{lemma} 

\begin{proof}
We first notice that under the assumptions that $\Dr$ and $f$ are smooth, we have that $u$ is also smooth up to the boundary $\partial\Dr$. Thus we can use the multiplication technique from \cite{deevpr} in order to obtain the $L^p$ estimate of $G'(|\nabla u|^2)$. For the sake of simplicity we set $\sigma=G'(|\nabla u|^2)$. We first test the equation \eqref{eqGaldogigievans} with the function $\sigma^{r-1}u\in H^1_0(\Dr)$ obtaining 
\be\label{eqride}
\begin{split}
\int_\Dr\sigma^r|\nabla u|^2\,dx+(r-1)\int_\Dr u\sigma^{r-1}\nabla u\cdot\nabla\sigma\,dx
&=\int_\Dr fu\sigma^{r-1}\,dx\\
&\le\|u\|_{L^\infty}\|f\|_{L^r}\left(\int_\Dr\sigma^r\,dx\right)^{1-1/r}.
\end{split}
\ee
Our main objective is to estimate the second term of the left-hand side. Multiplying \eqref{eqGaldogigievans} by $\dive(\sigma^{r-1}\nabla u)=(\sigma^{r-1}u_j)_j$ we get
\be\label{vasco}
\begin{split}
\int_\Dr(\sigma u_i)_i (\sigma^{r-1}u_j)_j\,dx
&=-\int_\Dr f(\sigma^{r-1}u_j)_j\,dx=-\int_\Dr f(\sigma^{r-2}\sigma u_j)_j\,dx\\
&=-\int_\Dr f\sigma^{r-2}(\sigma u_i)_i\,dx-(r-2)\int_\Dr f\sigma^{r-2}\sigma_ju_j\,dx\\
&\le\int_\Dr f^2\sigma^{r-2}\,dx+(r-2)\int_\Dr|f|\,\sigma^{r-2}|\nabla\sigma\cdot\nabla u|\,dx.
\end{split}
\ee
On the other hand, we can integrate by parts the left-hand side getting 
\[\begin{split}
\int_\Dr(\sigma u_i)_i(\sigma^{r-1}u_j)_j\,dx
&=-\int_\Dr\sigma u_i(\sigma^{r-1}u_j)_{ji}\,dx+\int_{\partial\Dr}(\sigma^{r-1}u_j)_j\sigma u_i n_i\,d\HH^{d-1}\\
&=\int_\Dr(\sigma u_i)_j(\sigma^{r-1}u_j)_i\,dx+\int_{\partial\Dr}\Big[(\sigma^{r-1}u_j)_j\sigma u_i n_i-\sigma u_i(\sigma^{r-1}u_j)_i n_i\Big]\,d\HH^{d-1}\\
&=\int_\Dr(\sigma u_i)_j(\sigma^{r-1}u_j)_i\,dx+\int_{\partial\Dr}\sigma^r\big(u_{jj}u_i n_i-u_{ij}u_i n_j\big)\,d\HH^{d-1}\\
&=\int_\Dr(\sigma u_i)_j(\sigma^{r-1}u_j)_i\,dx+\int_{\partial\Dr}\sigma^r u_n (\Delta u-u_{nn})\,d\HH^{d-1},
\end{split}\]
where $u_n$ and $u_{nn}$ indicate the first and the second derivatives in the direction of the exterior normal $n$ to $\partial\Dr$. Developing the term in the volume integral we have 
\[\begin{split}
\int_\Dr(\sigma u_i)_i(\sigma^{r-1}u_j)_j\,dx
&=\int_\Dr\Big(\sigma^r\|Hess(u)\|_2^2+(r-1)\sigma^{r-2}|\nabla\sigma\cdot\nabla u|^2+r\sigma^{r-1}\sigma_j u_i u_{ij}\Big)\,dx\\
&\qquad\qquad+\int_{\partial\Dr}\sigma^r u_n(\Delta u-u_{nn})\,d\HH^{d-1},
\end{split}\]
where $\|Hess(u)\|_2^2=\sum_{i,j}u_{ij}^2\le 0$. We first notice that $\sigma_j=2u_k u_{kj} G''(|\nabla u|^2)$ and thus, by the convexity of $G$ we have $\sigma_j u_i u_{ij}\ge 0$. On the other hand, on the boundary $\partial \Dr$ we have 
$$\Delta u=u_{nn}+Hu_n,$$
where $H$ is the mean curvature. Thus we have the inequality 
\[\begin{split}
\int_\Dr(\sigma u_i)_i(\sigma^{r-1}u_j)_j\,dx
&\ge\int_\Dr(r-1)\sigma^{r-2}|\nabla\sigma\cdot\nabla u|^2\,dx+\int_{\partial\Dr}H\sigma^r|\nabla u|^2\,d\HH^{d-1},
\end{split}\]
which together with \eqref{vasco} gives
\[\begin{split}
\int_\Dr\sigma^{r-2}|\nabla\sigma\cdot\nabla u|^2\,dx
&\le\int_\Dr|f|^2\sigma^{r-2}\,dx-\int_{\partial\Dr}H\sigma^r|\nabla u|^2\,d\HH^{d-1}\\
&\le\|f\|_{L^r}^2\|\sigma\|_r^{r-2}-\int_{\partial\Dr}H\sigma^r|\nabla u|^2\,d\HH^{d-1}.
\end{split}\]
Using \eqref{eqride} we can now repeatedly use the Young inequality $A^\alpha B^\beta \le \eps\alpha A+\eps^{\alpha/\beta}\beta B$ for $0\le \alpha,\beta\le 1$ and $\alpha+\beta=1$ obtaining
\[\begin{split}
\int_\Dr\sigma^r|\nabla u|^2\,dx
&\le(r-1)\int_\Dr\sigma^{r-1}\|u\|_{L^\infty}|\nabla u\cdot\nabla\sigma|\,dx+\|u\|_{L^\infty}\|f\|_{L^r}\|\sigma\|_{L^r}^{r-1}\\
&\le\eps\int_\Dr\sigma^r\,dx+\frac{(r-1)^2\|u\|_{L^\infty}^2}\eps\int_\Dr\sigma^{r-2}|\nabla u\cdot\nabla\sigma|^2\,dx+\|u\|_{L^\infty}\|f\|_{L^r}\|\sigma\|_{L^r}^{r-1}\\
&\le\eps\int_\Dr\sigma^r\,dx+\frac{(r-1)^2\|u\|_{L^\infty}^2}\eps\|f\|_{L^r}^2\|\sigma\|_r^{r-2}+\|u\|_{L^\infty}\|f\|_{L^r}\|\sigma\|_{L^r}^{r-1}\\
&\qquad\qquad\qquad-\frac{(r-1)^2\|u\|_{L^\infty}^2}\eps\int_{\partial\Dr}H\sigma^r|\nabla u|^2\,d\HH^{d-1}\\
&\le3\eps\int_\Dr\sigma^r\,dx+\Big(\frac{(r-1)^r}{\eps^{r-1}}+\frac1{\eps^{2r-1}}\Big)\|u\|_{L^\infty}^r\|f\|_{L^r}^r\\
&\qquad\qquad\qquad-\frac{(r-1)^2\|u\|_{L^\infty}^2}\eps\int_{\partial\Dr}H\sigma^r|\nabla u|^2\,d\HH^{d-1},
\end{split}\]
which is precisely the claim.
\end{proof}

In order to give a uniform estimate for $\theta_p$ on every domain satisfying the external ball condition and not only on the regular ones, we use an approximation argument. The following lemma that we use is well known in the $\gamma$-convergence theory.

\begin{lemma}\label{approxlemma05}
Suppose that $\Dr\subset\R^d$ is a bounded open satisfying the external ball condition and that $\Dr_n$ is a sequence of open sets of finite measure containing $\Dr$ and such that $|\Dr_n\setminus\Dr|\to0$. Let $f\in L^2(\Dr_1)$, $p\in[1,+\infty)$ be fixed and let $u_p$ be the minimizer of $F_p$ on $\Dr$ and $u_p^n$ on $\Dr_n$. Then the sequence $u_p^n$ converges to $u_p$ strongly in $H^1(\R^d)$ and $W^{1,2q}(\R^d)$. 
\end{lemma}

\begin{proof}
By the Sobolev inequality it is immediate to check that the sequence $u_n^p$ is uniformly bounded in $H^1(\R^d)$ and $W^{1,2q}(\R^d)$ and so we may suppose that up to a subsequence it converges in $L^2$ to a function $u$. Moreover, all the functions $u_p^n$ are uniformly bounded in $L^\infty$ and so, the function $u$ is zero almost everywhere on $\R^d\setminus\Dr$. Now, by \cite{deve} we have that $u\in H^1_0(\Dr)$. Moreover, for any function $v\in H^1_0(\Dr)\subset H^1_0(\Dr_n)$ we have
$$F_p(u)\le \liminf_{n\to\infty} F_p(u_p^n)\le F_p(v),$$
which proves that $u=u_p$ is the minimizer of $F_p$ on $\Dr$. Moreover, since $F_p(u)\ge F_p(u_p^n)$ for every $n$, we get that $F_p(u_p^n)\to F_p(u)$. By the strong $L^2$ convergence of $u_p^n$ to $u_p$ we obtain also that $\|\nabla u_p^n\|_{L^2}\to\|\nabla u_p\|_{L^2}$ and $\|\nabla u_p^n\|_{L^{2q}}\to\|\nabla u_p\|_{L^{2q}}$, which concludes the proof.
\end{proof}

\begin{prop}\label{mainest04}
Let $\Dr\subset\R^d$ be a bounded open set of finite perimeter satisfying the external ball condition and let $f\in L^\infty(\R^d)$. Then, for every $r\ge d$, there are constants $\delta>0$, depending on $\Dr$ and $f$, and $C$, depending on $r$, on the dimension $d$, the perimeter $P(\Dr)$, the diameter $\text{diam}(\Dr)$, the radius $R$ of the external ball, the norms $\|f\|_{L^\infty(\Dr)}$ and $\|u_1\|_{L^\infty}$, such that 
$$\|\theta_p\|_{L^r(\Dr)}\le C\quad\text{for every}\quad p\in(1,1+\delta).$$
\end{prop}

\begin{proof}
Suppose first that $\Dr$ has smooth boundary. We recall the notations $u_p$ for the minimizer of the functional $F_p$ in $H^1_0(\Dr)$, $\theta_p$ for the optimal reinforcement, and $C_p=\|\nabla u_p\|_{2q}^{-2(q-1)}$, where $q=p/(p-1)$. Setting
$$G_p(t)=t+\frac{mC_p}{q}t^q$$
we have that $u_p$ is the minimizer of the functional
$$H^1_0(\Dr)\ni u\mapsto \frac12\int_{\Dr}G_p(|\nabla u|^2)\,dx-\int_{\Dr} uf\,dx.$$
Moreover, we have
$$\theta_p=m\left( \frac{|\nabla u_p|}{\|\nabla u_p\|_{L^{2q}}}\right)^{2(q-1)}=mC_p |\nabla u_p|^{2(q-1)}=G_p'(|\nabla u_p|^2)-1.$$
By Lemma \ref{apriorilemma} and the mean curvature estimate $H_{\partial\Dr}\ge -1/R$, we have that
\be\label{startineq05}
\begin{split}
\int_\Dr\sigma_p^r|\nabla u_p|^2\,dx\le3\eps\int_\Dr\sigma_p^r\,dx
&+\Big(\frac{(r-1)^r}{\eps^{r-1}}+\frac1{\eps^{2r-1}}\Big)\|u_p\|_{L^\infty}^r\|f\|_{L^r}^r\\
&+\frac{(r-1)^2\|u_p\|_{L^\infty}^2}{\eps R}\int_{\partial\Dr}\sigma_p^r|\nabla u_p|^2\,d\HH^{d-1},
\end{split}
\ee
We denote by $S$ and $B$ the sets 
$$S=\{x\in\overline \Dr\ :\ |\nabla u_p(x)|\le \|\nabla u_p\|_{L^{2q}}\}\qquad\hbox{and}\qquad B=\{x\in\overline \Dr\ :\ |\nabla u_p(x)|> \|\nabla u_p\|_{L^{2q}}\},$$
and we notice that we have the inequality $\ds \sigma_p\le 1+m$ on $S$, where we set for simplicity $\sigma_p=G_p(|\nabla u_p|^2)$. We now estimate separately the three terms on the right-hand side. For the first one, after integrating separately on $S$ and $B$, we get
\be\label{ftrhs05}
\int_{\Dr}\sigma_p^r\,dx=\int_{S}\sigma_p^r\,dx+\int_B\sigma_p^r\,dx
\le(1+m)^r |\Dr|+\|\nabla u_p\|_{L^{2q}}^{-2}\int_B\sigma_p^r|\nabla u_p|^2\,dx.
\ee
We can estimate the second term simply by using Lemma \ref{inftyestlem04}, while for the third one we have 
\[\begin{split}
\int_{\partial\Dr}\sigma_p^r|\nabla u_p|^2\,d\HH^{d-1}
&=\int_{S\cap\partial\Dr}\sigma_p^r|\nabla u_p|^2\,d\HH^{d-1}+\int_{B\cap\partial\Dr}\sigma_p^r|\nabla u_p|^2\,d\HH^{d-1}\\
&\le(1+m)^{r-2}\int_{S\cap\partial\Dr}\sigma_p^2|\nabla u_p|^2\,d\HH^{d-1}\\
&\qquad+\|\nabla u_p\|_{L^{2q}}^{2-r}\int_{B\cap\partial\Dr}\sigma_p^r|\nabla u_p|^r\,d\HH^{d-1}\\
&\le(1+m)^{r-2}P(D)\|f\|_{L^\infty}^2\diam(\Dr)^2\left(1+\frac{\diam(\Dr)}{\rho}\right)^{2(d-1)}\\
&\qquad+\|\nabla u_p\|_{L^{2q}}^{2-r}P(D)\|f\|_{L^\infty}^r\diam(\Dr)^r\left(1+\frac{\diam(\Dr)}{\rho}\right)^{r(d-1)}.
\end{split}\]
Now taking in \eqref{startineq05} $\eps=\|\nabla u_p\|_{L^{2q}}^2/6$ and absorbing part of the first term in the right-hand side of \eqref{startineq05} into the left-hand side we get
\[\begin{split}
\frac12\int_{\Dr}\sigma_p^r|\nabla u_p|^2\,dx&\le\frac12\|\nabla u_p\|_{L^{2q}}^2(1+m)^r|\Dr|\\
&+6^{2r-1}\Big(\frac{(r-1)^r}{\|\nabla u_p\|_{L^{2q}}^{2r-2}}+\frac1{\|\nabla u_p\|_{L^{2q}}^{4r-2}}\Big)\|u_p\|_{L^\infty}^r\|f\|_{L^r}^r\\
&+\frac{6(r-1)^2\|u_p\|_{L^\infty}^2}{\|\nabla u_p\|_{L^{2q}}^2R}\Big[(1+m)^{r-2}P(D)\|f\|_{L^\infty}^2\diam(\Dr)^2\left(1+\frac{\diam(\Dr)}{\rho}\right)^{2(d-1)}\\
&\qquad+\|\nabla u_p\|_{L^{2q}}^{2-r}P(D)\|f\|_{L^\infty}^r\diam(\Dr)^r\left(1+\frac{\diam(\Dr)}{\rho}\right)^{r(d-1)}\Big].
\end{split}\]
We now notice that the same inequality holds for irregular $\Dr$. In fact we can approximate $\Dr$ by a sequence of smooth sets $\Dr_n$ containing $\Dr$ such that $|\Dr_n\setminus\Dr|\to 0$, $P(\Dr_n)\le 2 P(\Dr)$, $\text{diam}(\Dr_n)\le 2 \text{diam}(\Dr)$. By Lemma \ref{approxlemma05} the right-hand side of the above estimate passes to the limit, while the left-hand side is semicontinuous which finally gives the estimate passes to the limit. We now consider $\delta>0$ such that 
$$\frac12\|u_1\|_{L^\infty(\Dr)}\le \|u_1\|_{L^\infty(\Dr)}\le \frac32\|u_1\|_{L^\infty(\Dr)},\quad\hbox{for every}\quad p\in(1,1+\delta).$$
Then there is a constant $C=C(d,r,m,\|f\|_{L^\infty}, P(\Dr),\text{diam}(\Dr), \|u_1\|_{L^\infty})$ such that 
$$\int_{\Dr} \sigma_p^{r}|\nabla u_p|^2\,dx\le C,\quad\hbox{for every}\quad p\in(1,1+\delta).$$
Now the conclusion follows by the inequality $\int_\Dr\theta_p^r\,dx\le\int_\Dr\sigma_p^r\,dx$ and by applying one more time the estimate \eqref{ftrhs05}.
\end{proof}

\medskip

{\it Uniform $L^\infty$ estimate for $\theta_p$ in the case of constant force.} In the case $f=const$ on $\Dr$ a uniform estimate on $\theta_p$ can be obtained in a more direct way by using the Payne-Philippin maximum principle \cite{paph} that we recall below. The idea is similar to the one used by Kawohl in \cite{kawohl} for the infinity Laplacian.

\begin{teo}[Payne-Philippin maximum principle \cite{paph}]
Let $\Dr\subset\R^d$ be a bounded open set with $C^{2+\eps}$ boundary and let $g:[0,+\infty)\to[0,+\infty)$ be a $C^2$ function such that $2sg'(s)+g(s)>0$ on $[0,+\infty)$. Then, the function
$$\widetilde G(x)=\int_0^{|\nabla u (x)|^2}\big(2sg'(s)+g(s)\big)\,ds,$$
achieves its maximum on the boundary $\partial\Dr$, where the function $u$ is the solution to the equation 
$$-\dive\big(g(|\nabla u|^2)\nabla u\big)=1\quad\hbox{in}\quad \Dr,\qquad u=0\quad\hbox{on}\quad\partial\Dr.$$
\end{teo}

We are now in a position to give a uniform estimate on $\|\theta_p\|_{L^\infty(\Dr)}$.

\begin{prop}\label{propestextball}
Let $\Dr\subset\R^d$ be a bounded open set satisfying an external ball condition with radius $R\in (0,+\infty]$, where by $R=+\infty$ is intended that $\Dr$ is convex. Then, there is $\eps>0$ such that for $p\in(1,1+\eps]$ we have the estimate 
$$\theta_p\le2\|\nabla u_1\|_{L^\infty(\Dr)}^{-2}M^2\diam(\Dr)^2\left(1+\frac{\diam(\Dr)}{R}\right)^{2(d-1)}.$$
For convex $\Dr$ we have 
\be\label{estconvex}
\theta_p\le2\|\nabla u_1\|_{L^{\infty}}^{-2}M^2\diam(\Dr)^2.
\ee
\end{prop}

\begin{proof}
We consider only the case of $\Dr$ convex since the two cases are analogous. We first prove \eqref{estconvex} for regular domains and we then argue by approximation. Indeed, suppose first that $\Dr$ has regular $C^{2,\eps}$ boundary. Then $u_p$ is a solution of 
$$-\dive\big(g_p(|\nabla u_p|^2)\nabla u_p\big)=M\quad\hbox{in }\Dr,\qquad u_p=0\quad\hbox{on }\partial\Dr,$$
where $g_p$ is given by 
$$g_p(s)=s+mC_p s^{q-1},$$
and is sufficiently smooth when $q>1$ is large (i.e. $p$ close to $1$). By the Payne-Philippin theorem we have that the function 
$$\widetilde G_p(x)=\frac32|\nabla u_p|^2+\frac{2q-1}{q}mC_p|\nabla u_p|^{2q},$$
assumes its maximum on the boundary of $\Dr$. On the other hand, by \eqref{estconvex} on $\partial\Dr$ we have 
$$\widetilde G_p\le \frac{2q-1}{q}|\nabla u|^2\Big(1+mC_p|\nabla u|^{2(q-1)}\Big)\le \frac{2q-1}{q}M^2\diam(\Dr)^2,$$
and so on the entire $\Dr$ we have the estimate
$$mC_p|\nabla u|^{2q}\le \frac{q}{2q-1}G_p\le M^2\diam(\Dr)^2,$$
The optimal potential $\theta_p$ satisfies
$$\theta_p^p=\theta_p^{q/(q-1)}=\Big(mC_p|\nabla u_p|^{2(q-1)}\Big)^{q(q-1)}=\Big(m^{p/q}C_p^{1/(q-1)}\Big) m C_p|\nabla u_p|^{2q}.$$
We now recall that by the definition of $C_p$ we get
$$C_p^{1/(q-1)}=\Big(\|\nabla u_p\|_{L^{2q}}^{-2q/p}\Big)^{1/(q-1)}=\|\nabla u_p\|_{L^{2q}}^{-2}.$$
Thus we obtain
$$\theta_p\le \Big(m^{p/q}\|\nabla u_p\|_{L^{2q}}^{-2}M^2\diam(\Dr)^2\Big)^{1/p}.$$
The case of irregular $\Dr$ follows by approximation with smooth convex sets containing $\Dr$ and then applying Lemma \ref{approxlemma05}.
\end{proof}

\subsection{Conclusion of the proof of Theorem \ref{mainth04} and Theorem \ref{mainth04conv}}

\begin{proof}[Proof of Theorem \ref{mainth04}]
Let for $p>1$, $u_p\in W^{2q}_0(\Dr)$ be the minimizer of $F_p$ and let $\theta_p$ be the optimal potential relative to $u_p$. By the fact that $\theta_p\in\A_{m,p}$ and by Proposition \ref{mainest04} we have that there is a constant $\bar C$ such that for $p>1$ small enough we have
$$\int_\Dr \theta_p^p\,dx=m^p\qquad\text{and}\qquad \int_\Dr \theta_p^r\,dx\le C,$$
for some $r\ge d$. In particular, $\theta_p$ is uniformly bounded in $L^2(\Dr)$ and so, up to a subsequence, $\theta_p$ converges weakly in $L^2(\Dr)$ to a nonnegative function $\bar\theta\in L^2(\Dr)$. Moreover, we have
$$\qquad\int_\Dr\bar\theta\,dx=\lim_{p\to1}\int_\Dr\theta_p\,dx\le \liminf_{p\to1}\|\theta_p\|_{L^p}|\Dr|^{1/q}=m.$$
On the other hand, by Proposition \ref{convergenceprop} we have that as $p\to1$ the solutions $u_p$ converge strongly in $H^1_0(\Dr)$ to the minimum $u_1$ of the functional $F_1$. Thus, for every fixed smooth function with compact support $\varphi\in C^\infty_c(\Dr)$ we have 
$$\int_{\Dr}f\varphi\,dx=\int_{\Dr}(1+\theta_p)\nabla u_p\cdot\nabla \varphi\,dx \xrightarrow[p\to1]{}\int_{\Dr}(1+\bar\theta)\nabla u_1\cdot\nabla \varphi\,dx.$$ 
Thus, we have that $u_1$ is in fact the solution of the equation
$$-\dive\big((1+\bar\theta)\nabla u_1\big)=f\quad\text{in }\Dr,\qquad u_1=0\quad\hbox{on }\partial\Dr,$$
and after integration by parts 
$$E_f(\bar\theta)=\frac12\int_{\Dr}(1+\bar\theta)|\nabla u_1|^2\,dx-\int_\Dr f u_1\,dx=-\frac12\int_{\Dr}f u_1\,dx.$$
By the convergence of $u_p$ to $u_1$ in $L^2(\Dr)$ and by Proposition \ref{convergenceprop} we have
$$E_f(\bar\theta)=-\frac12\int_{\Dr}f u_1\,dx=\lim_{p\to1}-\frac12\int_{\Dr}f u_p\,dx=\lim_{p\to1}E_f(\theta_p)=\lim_{p\to1}F_p(u_p)=F_1(u_1),$$
which proves that 
\be\label{emtheta04}
E_f(\bar\theta)=\min_{u\in H^1_0(\Dr)} F_1(u).
\ee
On the other hand we have the general min-max inequality
$$\sup_{\theta\in\A_m} E_f(\bar\theta)=\sup_{\theta\in\A_m}\ \min_{u\in H^1_0(\Dr)}J_f(\theta,u)\le\min_{u\in H^1_0(\Dr)}\ \sup_{\theta\in\A_m} J_{f}(\theta,u)=\min_{u\in H^1_0(\Dr)} F_1(u),$$
which concludes the proof that $\bar\theta$ is a solution to the problem \eqref{initialpb}.

Claim {\it (i)} follows by the fact that $\bar\theta\in L^r(\Dr)$ and that $r$ can be chosen arbitrary large.

Claim {\it (ii)} is in fact \eqref{emtheta04} and the claim {\it (iv)} follows directly by this equality.

The regularity of the minimizer $u_1$ of $F_1$ {\it (iii)} was proved by Evans in \cite{evansC11}.
\end{proof}

\begin{proof}[Proof of Theorem \ref{mainth04conv}]
We first consider the case $r<+\infty$. We notice that in the case of a convex set there is no boundary term in the De Pascale-Evans-Pratelli \cite{deevpr} estimate. Thus, for $f\in L^r(\Dr)$ there is a constant $C_r$ such that for every $p$ large enough
$$\|\theta_p\|_{L^r}\le C_r\big(1+\|f\|_{L^r}\big).$$
The rest of the proof follows as in the proof of Theorem \ref{mainth04}.
\end{proof}

\begin{oss}
By using methods of optimal transportation theory slightly finer results can be obtained (\cite{filippo}, private communication):
\begin{description}
\item[-] if $\Dr$ has the exterior ball condition and $f\in L^p(\Dr)$ with $p\in[1,+\infty]$, then $\bar\theta\in L^p(\Dr)$;
\item[-] if $\Dr$ is any bounded open set and $f\in L^p_{loc}(\Dr)$ with $p\in[1,+\infty]$, then $\bar\theta\in L^p_{loc}(\Dr)$.
\end{description}
\end{oss}

\section{Alternative approach to the optimal reinforcement problem}\label{alt04}

Let $\Dr\subset\R^d$ be a bounded open set with a smooth boundary. 
The right hand-side $f$ represents a given force on the membrane $D$. We allow $f$ to be a general signed measure $f\in\M(\Dr)$. A reinforcement of the membrane is given by a measure $\mu$ in the class $\M^+(\Dr)$ of nonnegative measures on $\Dr$. For a vertical displacement $u\in C^1_c(\Dr)$, the energy of the membrane subjected to the force $f$ and reinforced by $\mu$ is given by the functional
$$J_f(u,\mu):=\frac12\int_{\Dr}|\nabla u|^2\,dx+\frac12\int_\Dr |\nabla u|^2\,d\mu-\int_{\Dr}u\,df\;,$$
and we define the Dirichlet energy functional
$$E_f(\mu):=\inf\Big\{J_f(u,\mu)\ :\ u\in C^1_c(\Dr)\Big\}.$$
As we noticed in Remark \ref{fmeasure} for some measures $f$ and $\mu$ we may have $E_f(\mu)=-\infty$. This occurs for instance if the force $f\notin H^{-1}(\Dr)$ and $\mu$ is the Lebesgue measure on $\Dr$. However these cases are ruled out from our discussion, because we are interested in the maximal reinforcement of the membrane. In fact, for any measure $f$ there exists a measure $\mu$, corresponding to the transport density of a suitable Monge optimal transport problem, such that $E_f(\mu)>-\infty$ (see for instance \cite{boubut}). The optimization problem we are interested is the following:
\be\label{initialpb}
\max\Big\{E_f(\mu):\ \mu\in\A_m\Big\},
\ee
where the parameter $m>0$ represents the available quantity of reinforcement material and the admissible set $\A_m$ is given by
$$\A_m=\Big\{\mu\in\M^+(\Dr)\ :\ \mu(\Dr)\le m\Big\}.$$
A first existence result in the class of nonnegative measures $\mu$ is the following.

\begin{prop}\label{exitmu04}
Let $\Dr\subset\R^d$ be a bounded smooth open set and let $m>0$. Then for every signed measure $f\in\M(\Dr)$ the optimization problem \eqref{initialpb} admits a solution $\bar\mu\in\M^+(\Dr)$ with $\bar\mu(\Dr)=m$ and $\bar\mu(\partial\Dr)=0$. 
Moreover, we have that 
$$E_f(\bar\mu)=\inf_{u\in W^{1,\infty}_0(\bar\Dr)} \frac12\int_{\Dr}|\nabla u|^2\,dx+\frac{m}{2}\|\nabla u\|^2_{L^\infty}-\int_{\Dr}u\,df.$$
\end{prop}

\begin{proof}
We first notice that for every fixed $u\in C^1_c(\Dr)$ the map $\mu\mapsto J_f(u,\mu)$ is continuous for the weak$^*$ convergence. Hence the map $\mu\mapsto E_f(\mu)$, being an infimum of continuous maps, is weak$^*$ upper semicontinuous. The conclusion follows by the weak$^*$ compactness of the $\A_m$ and the fact seen above that there is at least $\mu$ such that $E_f(\mu)$ is finite. The last claim follows by the monotonicity of the functional $E_f$ and the fact that for every measure $\mu\in\A_m$ it holds $E_f(\mu)=E_f(\mu\lfloor\Dr).$

To prove the last claim we notice that the functional $J_f:\A_m\times C^1_c(\Dr)\to\R$ satisfies the assumptions:
\begin{itemize}
\item $J_f$ is concave and upper semicontinuous in the first variable with respect to the weak-$\ast$ convergence;
\item $J_f$ is convex in the second variable. 
\end{itemize}
According to a classical result (see for instance \cite{cla77}) we may interchange the $\inf$ and the $\sup$ for $J_f$ obtaining
\be\label{interch}
\sup_{\mu\in\A_m}\ \inf_{u\in C^1_c(\Dr)}J_f(\mu,u)=\inf_{u\in C^1_c(\Dr)}\ \sup_{\mu\in\A_m}J_f(\mu,u).
\ee
The supremum with respect to $\mu$ is easy to compute and we have
$$\sup_{\mu\in\A_m}J_f(\mu,u)=\frac12\int_{\Dr}|\nabla u|^2\,dx+\frac{m}{2}\|\nabla u\|^2_{L^\infty}-\int_{\Dr}u\,df$$
as required.
\end{proof}

Therefore, we end up with a variational problem for the functional
\be\label{varpb}
F_1(u)=\frac12\int_{\Dr}|\nabla u|^2\,dx+\frac{m}{2}\|\nabla u\|^2_\infty-\int_{\Dr}u\,df.
\ee
The existence and the regularity of the minimizers of $F_1$ has been widely studied in the literature (see for instance \cite{bresta} and \cite{evansC11}). We summarize the known results below.

\begin{teo}\label{regoth04}
For every measure $f\in\M(\Dr)$ the optimization problem 
$$\min\Big\{F_1(u)\ :\ u\in H^1_0(\Dr)\Big\}$$ 
admits a unique solution $\bar u$. Moreover, 
\begin{enumerate}[(i)]
\item if $f\in L^p(\Dr)$ with $1<p<+\infty$, then $u \in W^{2,p}(\Dr)$; in particular if $p>d$, then the solution $\bar u$ is $C^{1,\alpha}(\bar\Dr)$ for some $\alpha>0$;
\item If $\Dr$ is convex and $f\in L^\infty(\Dr)$ then $\Delta \bar u\in L^\infty(\Dr)$;
\item if $f\in C^2(\Dr)$, then $\bar u\in C^{1,1}(\Dr)$.
\end{enumerate}
\end{teo}

\begin{proof}
The existence follows by the direct methods of the calculus of variations while the uniqueness is a consequence of the strict convexity of the functional $F_1$. 
The solution $\bar u$ is clearly Lipschitz and setting $M=\|\nabla \bar u\|_{L^\infty}$ we have that $\bar u$ is also the solution of the problem 
$$\min\Big\{\frac12\int_{\Dr}|\nabla u|^2\,dx-\int_\Dr u\,df\ :\ u\in W^{1,\infty}_0(\Dr),\ \|\nabla u\|_{L^\infty}\le M\Big\}.$$
The claims {\it (i)} and {\it (ii)} were proved in \cite{bresta}, while {\it (iii)} was proved in \cite{evansC11}. 
\end{proof}

Let now $f\in L^p(\Dr)$ with $p>d$ and let $\bar\mu$ be the solution to the problem \eqref{initialpb}. We notice that $J_f(\bar\mu,\cdot)$ can be extended from a functional on $C^1_c(\Dr)$ to a functional on $C^1(\Dr)\cap C_0(\bar\Dr)$. Thus, we can use the minimizer $\bar u$ as a test function in $J_f(\bar\mu,\cdot)$ obtaining
$$F_1(\bar u)\ge J_f(\bar \mu,\bar u)\ge \inf_{u\in C^1(\Dr)\cap C_0(\bar\Dr)}J_f(\bar\mu,u)=E_f(\bar\mu)=F_1(\bar u),$$
where the last equality follows by \eqref{interch}. Thus the second inequality is an equality and implies that $\bar u$ is the minimizer of the functional $J_f(\bar\mu,\cdot)$ in $C^1(\Dr)\cap C_0(\bar\Dr)$ and so it satisfies the Euler-Lagrange equation 
\be\label{eulag456}
-\dive\big((1+\bar\mu)\nabla \bar u\big)=f\quad\text{in }\Dr,\qquad \bar u=0\quad\text{on }\partial\Dr.
\ee
Using $\bar u$ as a test function in \eqref{eulag456} we get 
$$\int_{\Dr}|\nabla \bar u|^2\,dx+\int_{\Dr}|\nabla \bar u|^2\,d\bar\mu=\int_\Dr \bar uf\,dx.$$
On the other hand, since $\bar u$ is a minimizer of $F_1$, we have that the function $t\mapsto F_1(t\bar u)$ has a minimum in $t=1$, which gives 
$$\int_{\Dr}|\nabla \bar u|^2\,dx+m\|\nabla \bar u\|_{L^\infty}^2=\int_\Dr \bar uf\,dx.$$
Putting together the two identities we have that 
$$\int_{\Dr}|\nabla \bar u|^2\,d\bar\mu=m\|\nabla \bar u\|_{L^\infty}^2=mM^2,$$
which together with \eqref{eulag456} gives the Monge-Kantorovich transport equation
\be\label{monka}
\begin{cases}
-\dive(\bar\mu\nabla \bar u)=f+\Delta \bar u\quad\text{in }\Dr,\qquad\bar u=0\quad\text{on }\partial\Dr,\\
|\nabla \bar u|\le M,\\
|\nabla \bar u|=M\quad\text{on }\spt(\bar\mu).
\end{cases}
\ee

This equation was widely studied in the framework of optimal transport theory. In particular, it was proved in \cite{deevpr}, \cite{fil} and \cite{depr} that the integrability properties of the right-hand side give the integrability of $\bar\mu$. Thus we can prove directly Theorem \ref{mainth04conv} using these optimal transport results.

\begin{proof}[Alternative Proof of Theorem \ref{mainth04conv}]
Let $f\in L^r(\Dr)$ for some $r>d$. By Proposition \ref{exitmu04} there is a solution $\bar\mu$ to the problem \eqref{initialpb}. By the regularity result of Brezis and Stampacchia \cite{bresta} the minimizer $\bar u$ of $F_1$ is $C^{1,\alpha}$ for some $\alpha>0$ and $\Delta \bar u\in L^r(\Dr)$. Thus $\bar\mu$ and $\bar u$ solve \eqref{monka}, whose right-hand side is in $L^r(\Dr)$. Now by the summability results from \cite{deevpr} and \cite{fil} we have that $\mu\in L^r(\Dr)$, which concludes the proof in the case $r<\infty$. If $f\in L^\infty$, then again by \cite{bresta} we have that $\Delta \bar u +f\in L^\infty$. Now by \cite{fil} (and also \cite{filippo}) we have that also $\bar\mu\in L^\infty(\Dr)$.
\end{proof}

\section{The elastic-plastic torsion problem}\label{sobst}

The solution of the optimization problem \eqref{initialpb} is related to the solution of the variational minimization problem \eqref{varpb}. In fact, the solution of \eqref{varpb} is precisely the torsion function $\bar u$ associated to the optimal elasticity term $\theta$. On the other hand, in the case $f\equiv 1$, the problem \eqref{varpb} is equivalent to the well-known of {\it elastic-plastic torsion problem}. 
%
In fact if $\bar u\in W^{1,\infty}(\bar\Dr)$ is the solution of the variational problem \eqref{varpb}, then setting $\kappa=\|\nabla \bar u\|_{L^\infty}$, it is immediate to check that $\bar u$ is also the solution of 
\be\label{obstgrad}
\min\Big\{\frac12\int_{\Dr}|\nabla u|^2\,dx-\int_{\Dr}u\,dx\ :\ u\in H^1_0(\Dr),\ \|\nabla u\|_{L^\infty}\le \kappa\Big\}.
\ee
This problem was a subject of an intense study in the past. The first results are due to Ting who studied the regularity of $\bar u$ and of the free boundary $\partial\{|\nabla \bar u|<\kappa\}$ in the case of a square \cite{ting1}, a regular \cite{ting2} and an irregular polygon \cite{ting3}. The key observation is that the problem \eqref{obstgrad} is equivalent to the following {\it obstacle problem}: 
\be\label{obstacle}
\min\Big\{\frac12\int_{\Dr}|\nabla u|^2\,dx-\int_{\Dr}u\,dx\ :\ u\in H^1_0(\Dr),\ u\le \kappa\, d_{\partial\Dr}\Big\},
\ee
where by $d_{\partial\Dr}:\Dr\to[0,+\infty)$ we denote the distance function
$$d_{\partial\Dr}(x)=\min\big\{|x-y|\ :\ y\in\partial\Dr\big\}.$$
For a general domain $\Dr\subset\R^d$ the equivalence of \eqref{obstgrad} and \eqref{obstacle} was proved by Brezis and Sibony \cite{bresib} and is based on a bound on the Lipschitz constant of the solution of \eqref{obstacle}. In fact since we have that 
$$\Big\{u\in H^1_0(\Dr)\ :\ |\nabla u|\le k\Big\}\subset\Big\{u\in H^1_0(\Dr)\ :\ u\le \kappa\,d_{\partial\Dr}\Big\},$$
it is sufficient to prove that the solution of the obstacle problem \eqref{obstacle} is $\kappa$-Lipschitz. This may not be true in the case of a general force term $f$ when large oscillations can be produced directly by $f$ and so the equivalence of the two problems does not hold in general (see \cite{caffri}).

The regularity of $\bar u$ is due to several authors. In \cite{bresta} Brezis and Stampacchia showed that $\bar u\in C^{1,\alpha}(\Dr)$, for every $\alpha\in(0,1)$. The optimal regularity $\bar u\in C^{1,1}(\Dr)$ was obtained by Caffarelli-Riviere \cite{cari}, Evans \cite{evansC11} and Wiegner \cite{wiegner}. 
In the case of planar domains with boundary which is a union of $C^3$ curves the free boundary was characterized by Caffarelli and Friedman \cite{caffri}. 

In this case the optimal reinforcement $\theta$ is supported on the set $\{\bar u=\kappa d_{\partial\Dr}\}$ and satisfies the Hamilton-Jacobi equation 
$$\nabla\theta\cdot\nabla d_{\partial\Dr}+\theta \Delta d_{\partial\Dr}+\frac1\kappa+\Delta d_{\partial\Dr}=0\quad\hbox{on }\{\bar u=\kappa d_{\partial\Dr}\}.$$
This fact was used by Brezis in \cite{brezis} to obtain $\theta$ explicitly and also to show that $\theta$ is continuous. We notice that the regularity of $\theta$ for a general nonconstant $f$ is an open question. We conclude this section with an example in which the optimal reinforcement $\theta$ can be easily computed; other cases are treated numerically in Section \ref{snum}.

\begin{exam}[The radial case]Let $\Dr$ be the unit ball of $\R^d$ and $f=1$. We are then dealing with the optimal torsion problem. Passing to polar coordinates gives the optimization problem
$$\min\Big\{\frac12\int_0^1r^{d-1}|u'|^2\,dr+\frac{m}{2d\omega_d}\|u'\|^2_\infty-\int_0^1r^{d-1}u\,dr\ :\ u(1)=0\Big\},$$
where $\omega_d$ is the measure of $\Dr$. After some elementary calculations we find that the solution $\bar u$ is given by
$$\bar u(r)=\begin{cases}
\ds\frac{a^2-r^2}{2d}+\frac{a(1-a)}{d}&\hbox{if }r\in[0,a]\\
\ds\frac{a}{d}(1-r)&\hbox{if }r\in[a,1]
\end{cases}$$
where $a=a_m$ is a suitable number in $]0,1[$. Optimizing with respect to $a$ we obtain that $a_m$ is the unique solution of
$$a^{d+1}-(d+1)a\Big(1+\frac{m}{\omega_d}\Big)+d=0.$$
It remains to compute the optimal function $\bar\theta$. By the Euler-Lagrange equation of
$$\int_{a_m}^1\frac{1+\theta}{2}r^{d-1}|u'|^2\,dr-\int_{a_m}^1r^{d-1}u\,dr$$
we obtain the expression of $\bar\theta$, given by
$$\bar\theta(r)=\frac{r}{a_m}-1.$$
\end{exam}

\section{Approximation of the free boundary problem}\label{snum}

\subsection{A discrete convex constrained formulation}

We introduce in this section, a convergent discretization of the problem

\be\label{eq_fb}
\min\Big\{\int_D\Big(\frac12|\nabla u|^2-fu\Big)\,dx+\frac{m}{2}\|\nabla u\|^2_\infty\ :\ u\in H^1_0(D)\Big\}
\ee
for a domain $D\subset\R^d $. We focus our description on the case of the dimension $2$. It is straightforward to generalize this approach to the three dimensional case. Our approach relies both on classical techniques in approximation of linear partial differential equations and on a suitable reformulation of this non-smooth convex energy. Assume $\tau_h$ is a given mesh made of simplexes which approximates the computational domain $D$. In the context of $P_1$ finite elements, we consider $H^1(D_h)$ the space of functions which are globally continuous and piecewise linear on every simplex of the mesh $\tau_h$. Analogously to the continuous framework, we denote by $H^1_0(D_h)$ the functions of $H^1(D_h)$ which vanish at every boundary point of $\tau_h$. We define $K_h$ and $M_h$ respectively the stiffness and the mass matrices associated to this discretization, more precisely we have
$$K_h=\left(\int_{D_h}\nabla u_i\cdot\nabla u_j\,dx\right)_{i,j}
\qquad\qquad
M_h=\left(\int_{D_h}u_i u_j\,dx\right)_{i,j}$$
where the family $(u_i)$ is a basis of the the vector space $H^1_0(D_h)$. Thus if we denote by $\lambda$ the coordinates of a function of $H^1_0(D_h)$ in this basis, we can introduce the discrete $E_m(\lambda)$ energy associated to problem \eqref{eq_fb} defined by
$$\frac12\lambda^T K_h \lambda_ - v_f^T M_h \lambda + \frac{m}{2} \sup_l \{ (K_x\lambda)_l^2 + (K_y\lambda)_l^2 \}$$
where the exponent $T$ denotes the transpose operator, the coordinates of the vector $v_f$ are given by the evaluation of the fixed function $f$ on every vertex of the mesh and the matrices $K_x$ and $K_y$ are the linear operator associated to the evaluation of partial derivatives $\partial_x$ and $\partial_y$ in $H^1(D_h)$. Notice that $K_x$ and $K_y$ are not square matrices since their first dimension is indexed by the number of simplexes of $\tau_h$. In order to get rid of the $\sup$ term which makes the original energy non smooth, we classically introduce a new variable $t \in \R$ and the new quadratic constraints
\be\label{consd}
(K_x\lambda)_l^2 + (K_y\lambda)_l^2 \leq t
\ee
where $l$ is an index associated to the set of the simplexes of $\tau_h$. In this context, we define a new equivalent smooth convex constrained optimization problem which is the minimization of the energy $F_m(\lambda,t)$ defined by
\be\label{costd}
F_m(\lambda,t)=\frac12\lambda^T K_h\lambda_ - v_f^T M_h\lambda+\frac{m}{2} t
\ee
for admissible $(\lambda,t)$ which satisfy the constrained (\ref{consd}).

\subsection{Numerical results in 2D and 3D}

In order to study the numerically solution of problem \ref{costd} under constraints \ref{consd}, we used a classical interior point solver provided by the so called IPOPT library available from COIN-OR (\url{http://www.coin-or.org}). Since the cost function and all the constraints are of quadratic type, we provided an exact evaluation of the Hessian matrix with respect to our variable $(\lambda,t)$. Moreover, noticed that due to this quadratic formulation we can easily define a fixed pattern for the Hessian matrix which allows, through factorization techniques, to speed up dramatically the computation of the descent direction at every step. We present below several experiments based on this numerical approximation.

We first focused our study to the example of the unite disk. We discretize this set by a mesh made of $2 \times 10^4$ triangles. This complexity leads to an optimization problem of an approximate number of unknowns of the same order and under $2\times 10^4$ constraints. We represent on the first column of Figure \ref{fig:f0} the optimal graphs obtained for $m = 0,\, 0.1,\, 0.5$. We added on the graphs the analytically computed radial solution and observe an almost perfect match between our solution and the theoretical one.
On the second column of Figure \ref{fig:f0}, we plot the graph of the gradient of the solution and for $m = 0.1,\, 0.5$. We underlined the free boundary set by adding a bolted line along its boundary. We reproduce the same experiments for three other geometries on Figures \ref{fig:f1}, \ref{fig:f2} and \ref{fig:f3}. As expected, we observe in the three cases that for $m$ large enough the free boundary of the problem converges to the singular set of the distance function of the domain.

Finally, to illustrate the efficiency of this approach even in a large scale context, we approximated the optimal profiles on three dimensional problems. In Figures \ref{fig:f4} and \ref{fig:f5} we discretized a cube and a dodecahedron by meshes made of approximately $10^6$ simplexes for $m=5$. The IPOPT library was able to solve such problems in several minutes on a standard computer. We plot in Figures \ref{fig:f4} and \ref{fig:f5} different cuts of the graph of the modulus of the gradient of the solution. As in the two dimensional case we observe a convergence of the free boundary to a complex geometrical set which is very close from the singular set of the distance function.

\begin{figure}
\centering
\begin{tabular}{r r}
\includegraphics[width=\widthhh cm ]{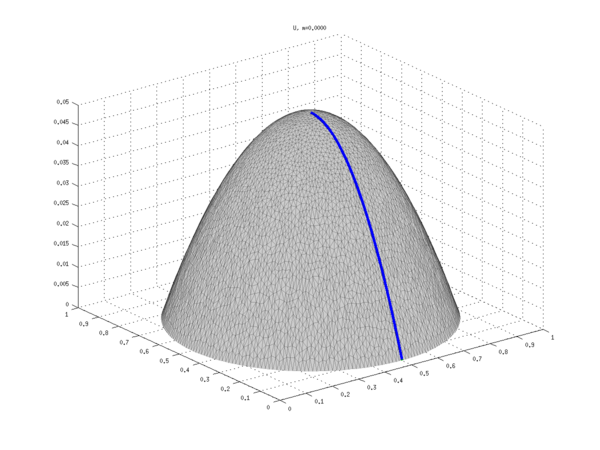}&
\includegraphics[width=\widthhh cm ]{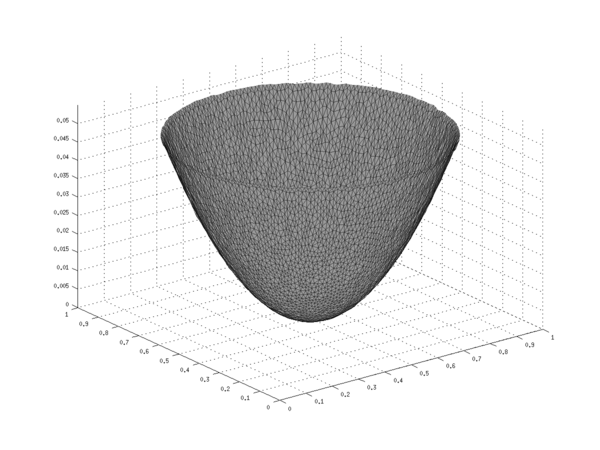}\\
\includegraphics[width=\widthhh cm ]{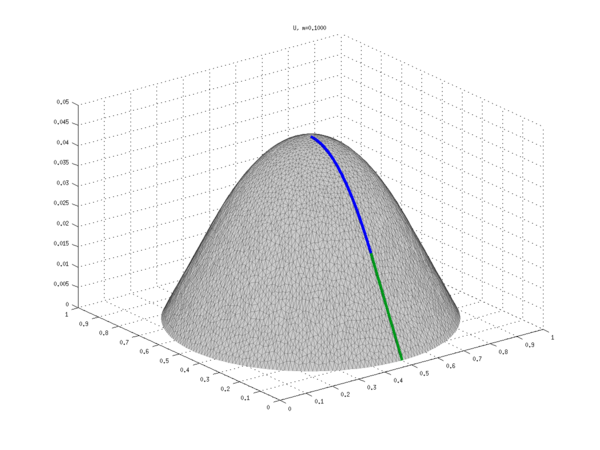}&
\includegraphics[width=\widthhh cm ]{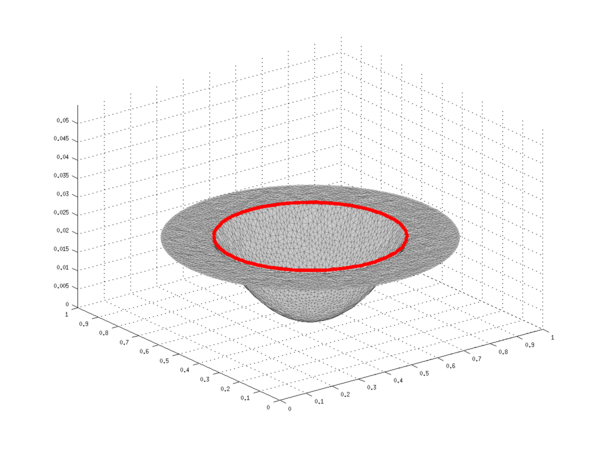}\\
\includegraphics[width=\widthhh cm ]{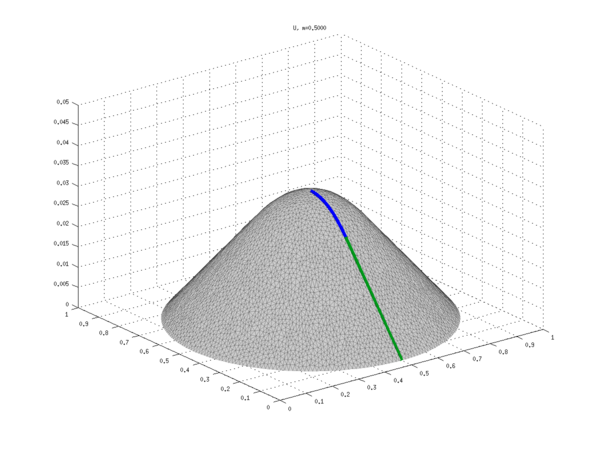}&
\includegraphics[width=\widthhh cm ]{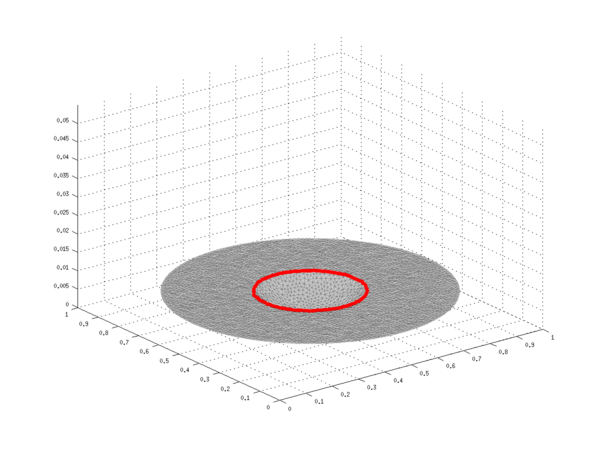}
\end{tabular}
\caption{Optimization on a disk for $m = 0, 0.1, 0.5$}
\label{fig:f0}
\end{figure}

\begin{figure}
\centering
\begin{tabular}{r r r}
\includegraphics[width=\widthhh cm ]{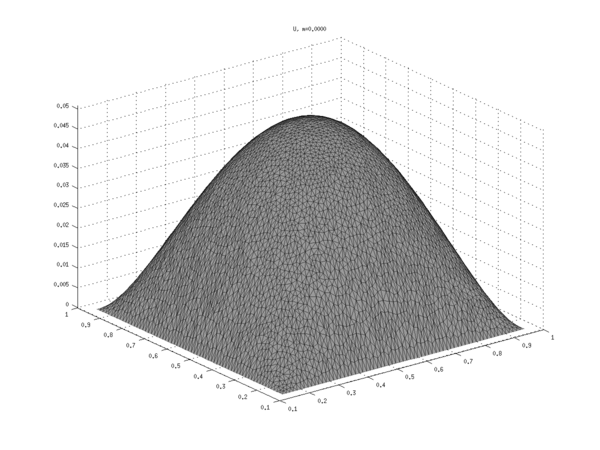}&
\includegraphics[width=\widthhh cm ]{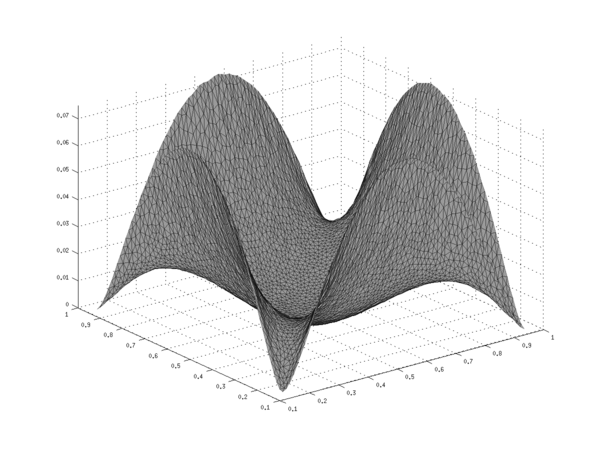}&\\
\includegraphics[width=\widthhh cm ]{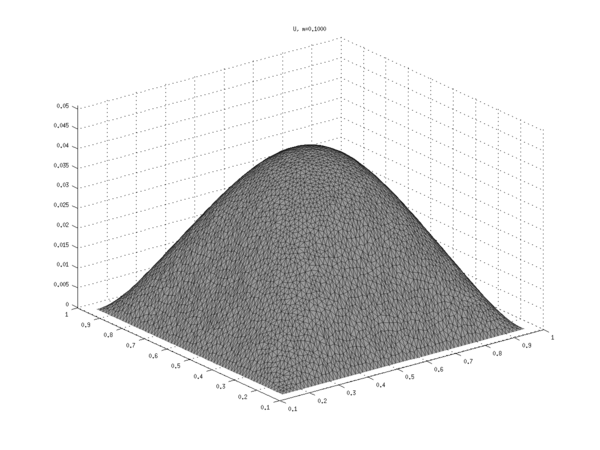}&
\includegraphics[width=\widthhh cm ]{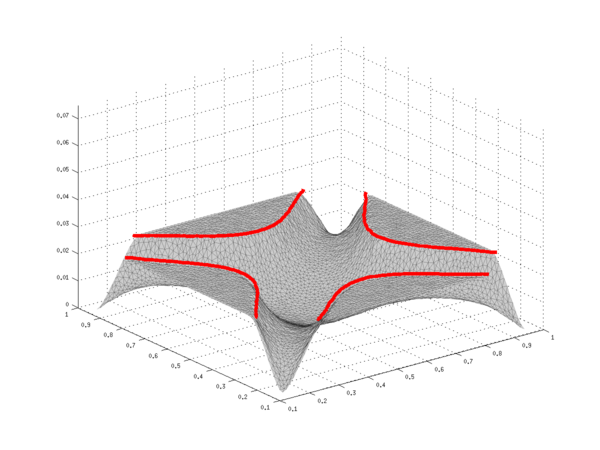}\\
\includegraphics[width=\widthhh cm ]{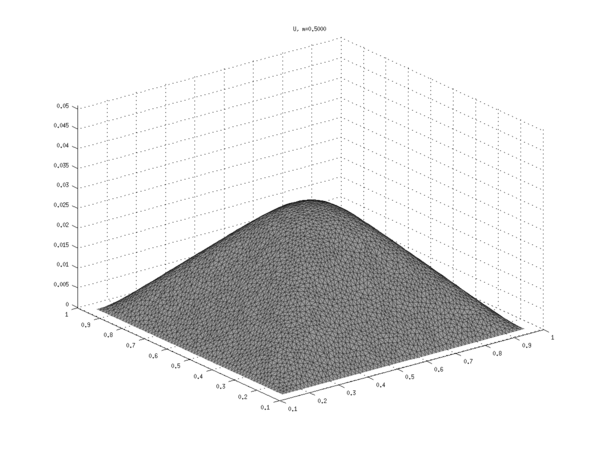}&
\includegraphics[width=\widthhh cm ]{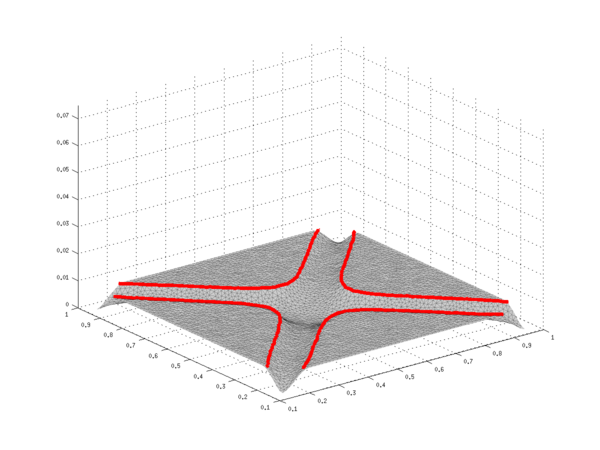}
\end{tabular}
\caption{Optimization on a square for $m = 0, 0.1, 0.5$}
\label{fig:f1}
\end{figure}


\begin{figure}
\centering
\begin{tabular}{r r r}
\includegraphics[width=\widthhh cm ]{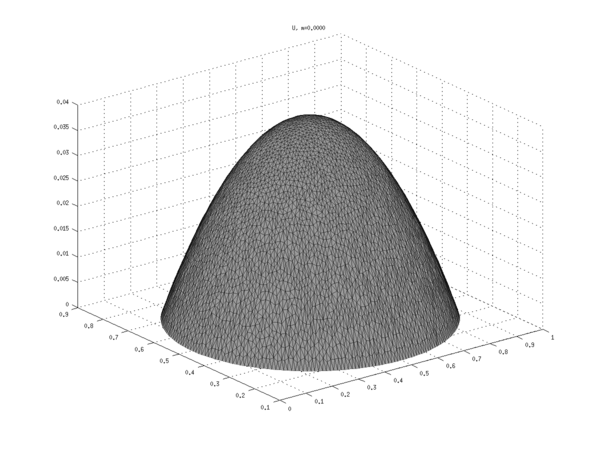}&
\includegraphics[width=\widthhh cm ]{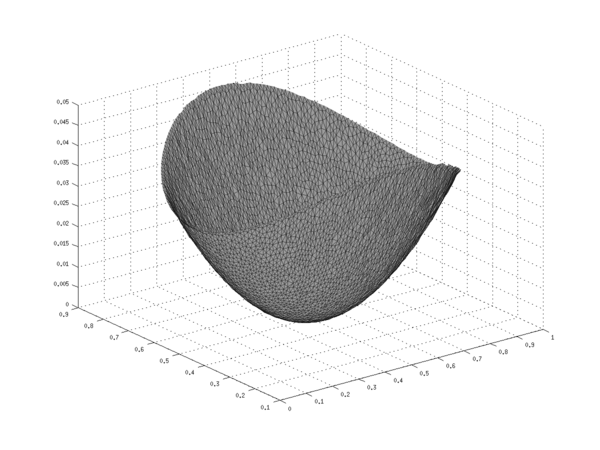}\\
\includegraphics[width=\widthhh cm ]{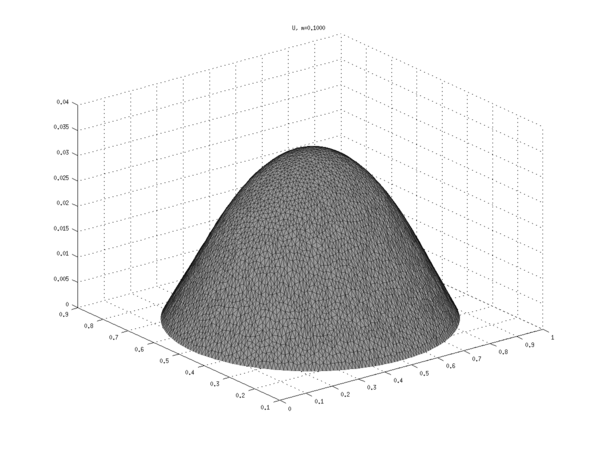}&
\includegraphics[width=\widthhh cm ]{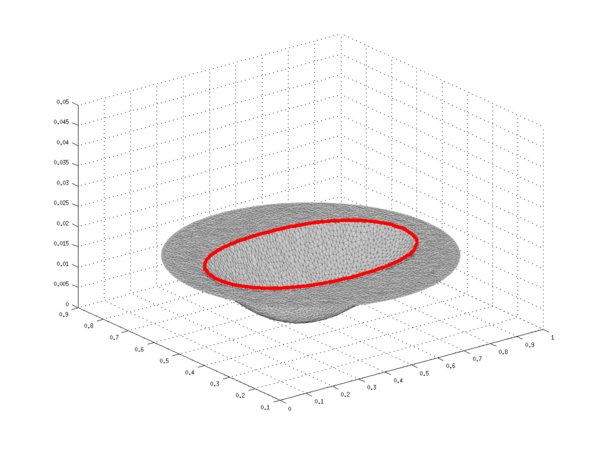}\\
\includegraphics[width=\widthhh cm ]{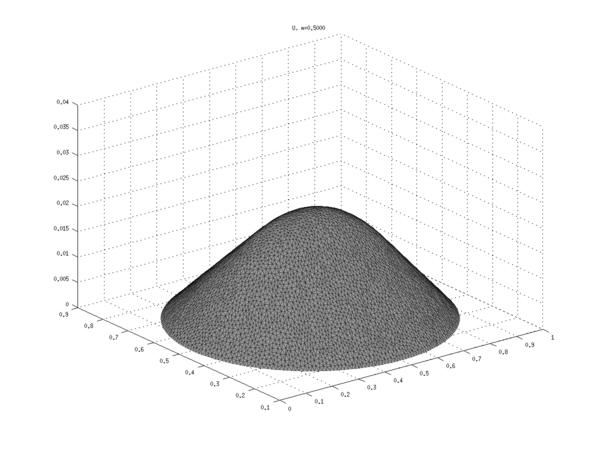}&
\includegraphics[width=\widthhh cm ]{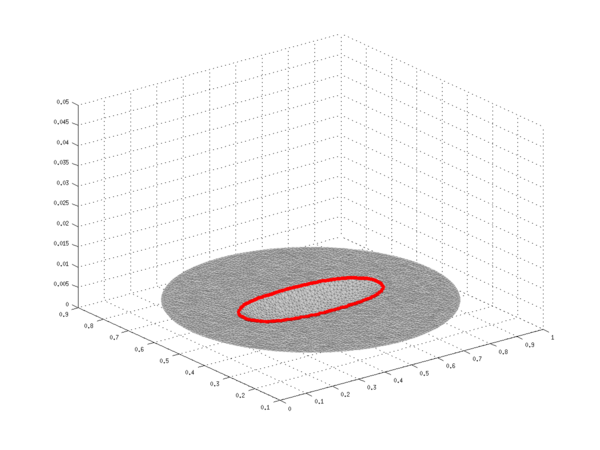}
\end{tabular}
\caption{Optimization on an ellipse for $m = 0, 0.1, 0.5$}
\label{fig:f2}
\end{figure}


\begin{figure}
\centering
\begin{tabular}{r r r}
\includegraphics[width=\widthhh cm ]{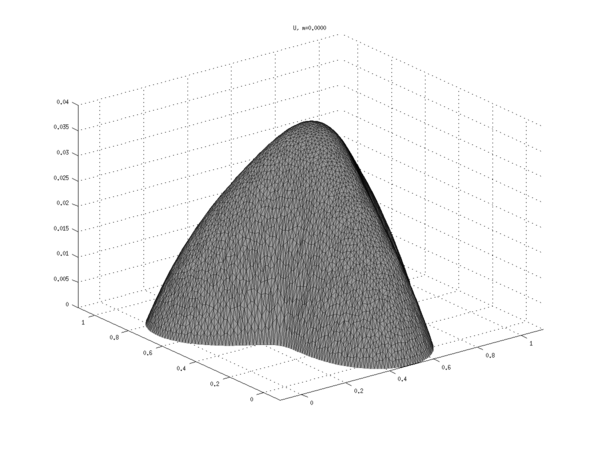}&
\includegraphics[width=\widthhh cm ]{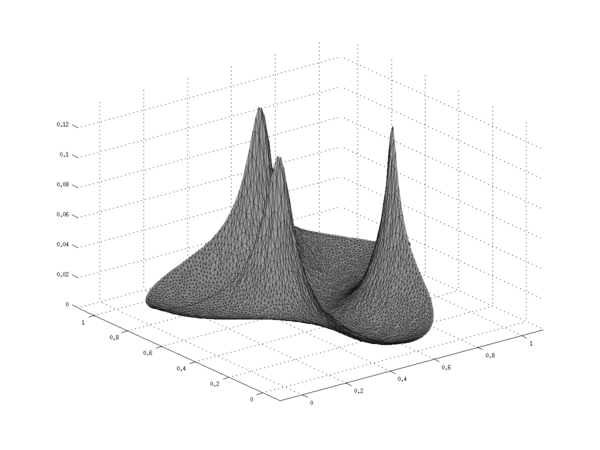}\\
\includegraphics[width=\widthhh cm ]{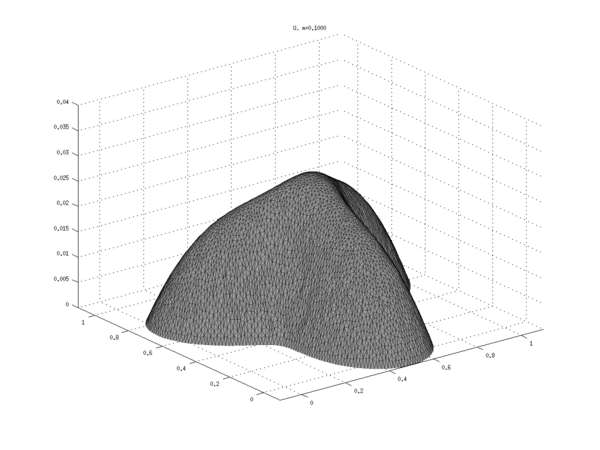}&
\includegraphics[width=\widthhh cm ]{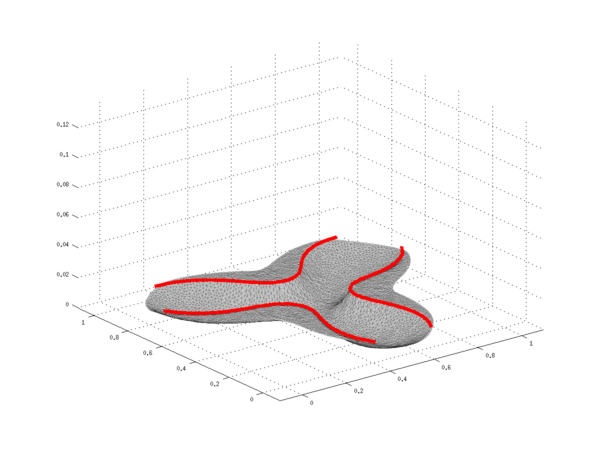}\\
\includegraphics[width=\widthhh cm ]{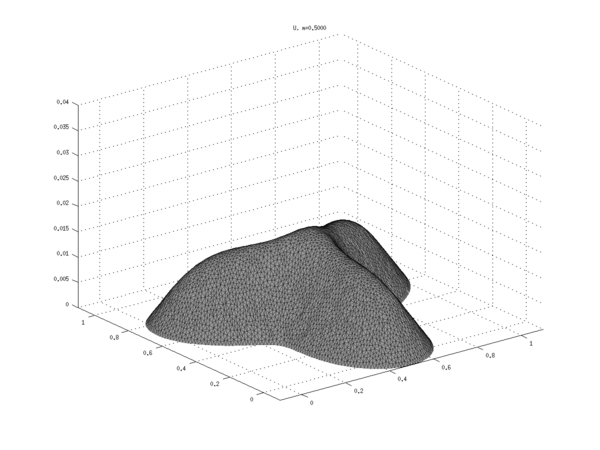}&
\includegraphics[width=\widthhh cm ]{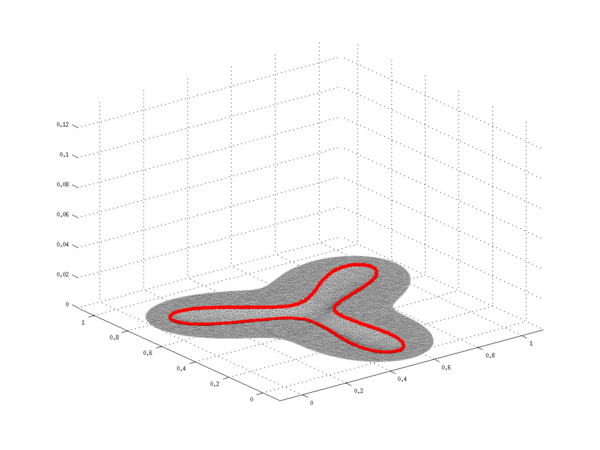}
\end{tabular}
\caption{Optimization on a treffle for $m = 0, 0.1, 0.5$}
\label{fig:f3}
\end{figure}


\begin{figure}
\centering
\begin{tabular}{r r r}
\includegraphics[width=\widthh cm ]{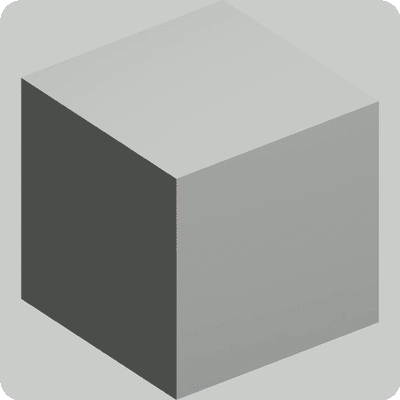}&
\includegraphics[width=\widthh cm ]{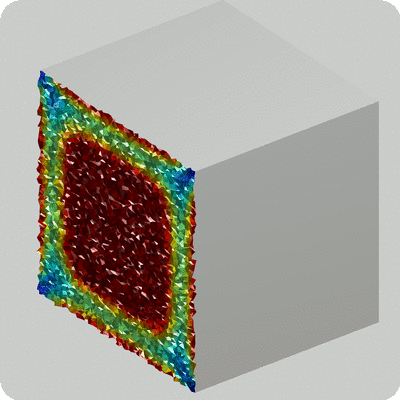}&
\includegraphics[width=\widthh cm ]{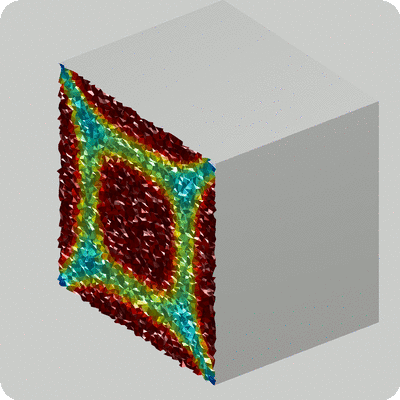}
\\
\includegraphics[width=\widthh cm ]{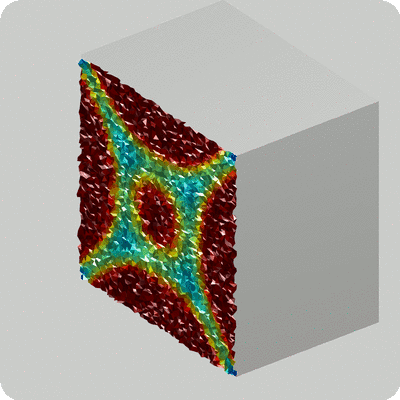}&
\includegraphics[width=\widthh cm ]{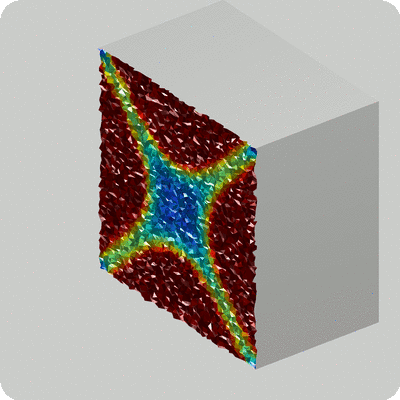}&
\includegraphics[width=\widthh cm ]{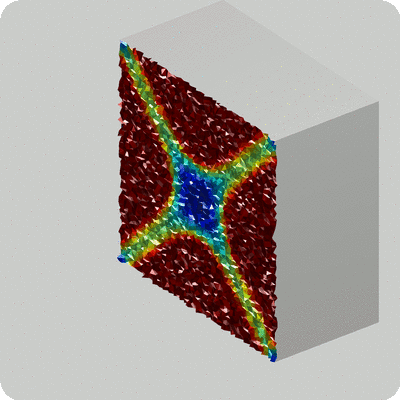}
\\
\includegraphics[width=\widthh cm ]{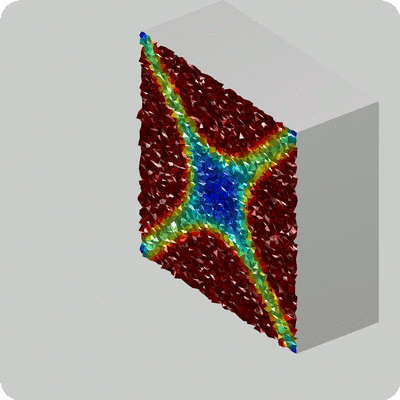}&
\includegraphics[width=\widthh cm ]{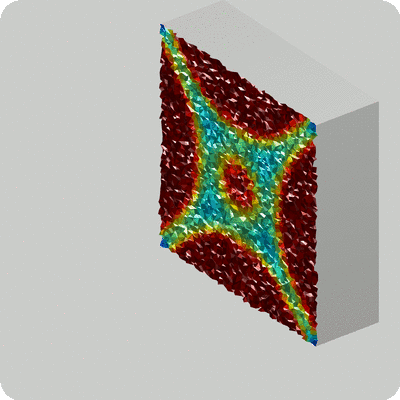}&
\includegraphics[width=\widthh cm ]{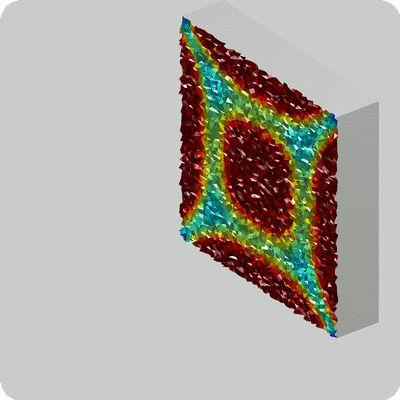}
\end{tabular}
\caption{Tomographic view of the optimal solution defined in a cube for $m=5$. The free boundary is visible in light color and is very close to the cut locus of the cube}
\label{fig:f4}
\end{figure}

\begin{figure}
\centering
\begin{tabular}{r r r}
\includegraphics[width=\widthh cm ]{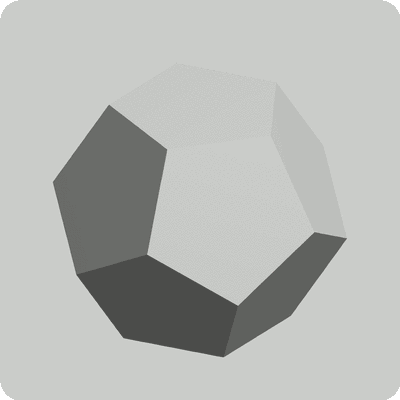}&
\includegraphics[width=\widthh cm ]{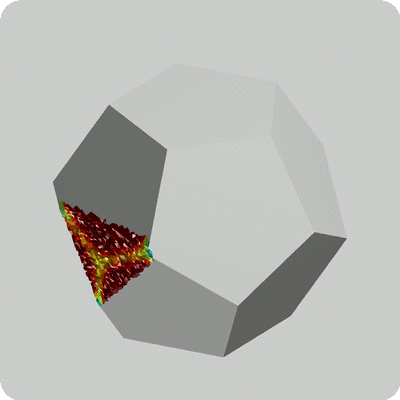}&
\includegraphics[width=\widthh cm ]{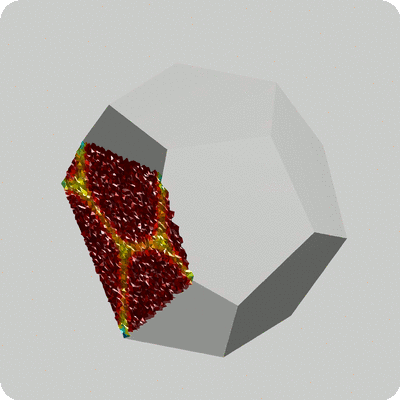}
\\
\includegraphics[width=\widthh cm ]{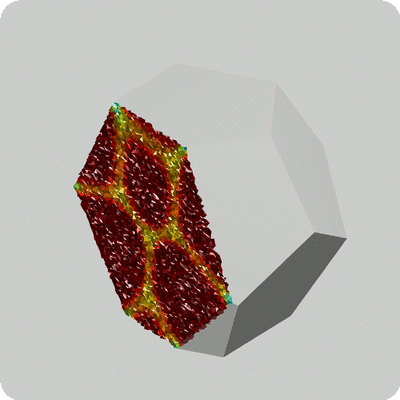}&
\includegraphics[width=\widthh cm ]{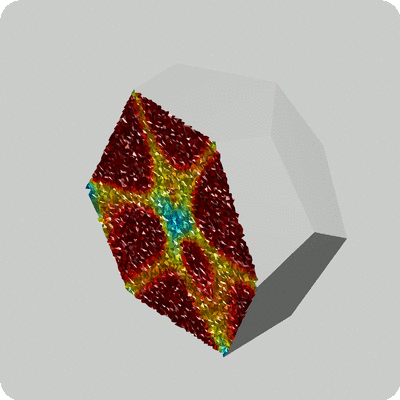}&
\includegraphics[width=\widthh cm ]{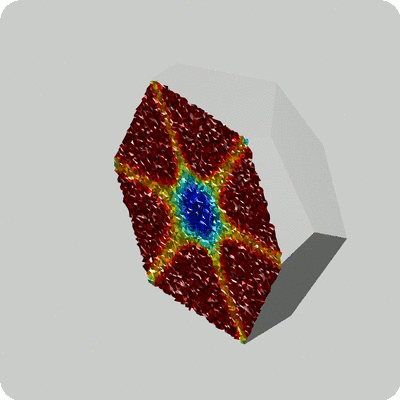}
\\
\includegraphics[width=\widthh cm ]{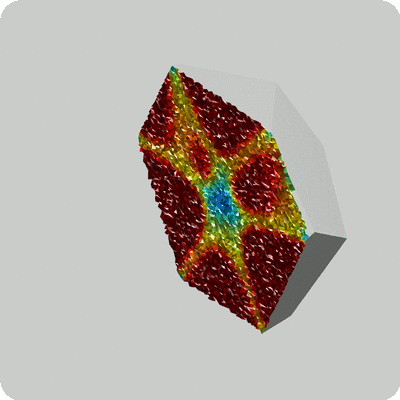}&
\includegraphics[width=\widthh cm ]{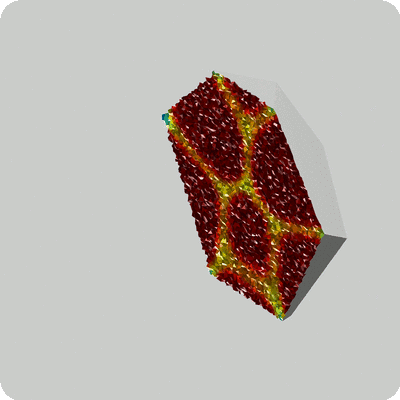}&
\includegraphics[width=\widthh cm ]{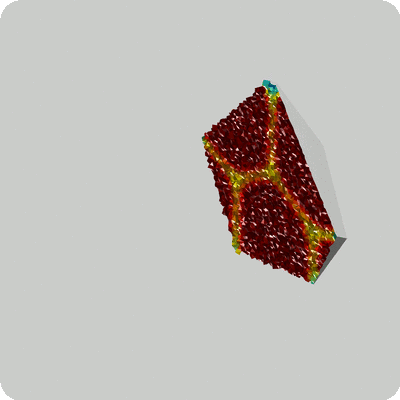}
\end{tabular}
\caption{Tomographic view of the optimal solution defined in a dodecahedron for $m=5$. The free boundary is visible in light color and is very close to the cut locus of the dodecahedron}
\label{fig:f5}
\end{figure}


\medskip
\ack A part of this paper was written during a visit of the authors at the Johann Radon Institute for Computational and Applied Mathematics (RICAM) of Linz. The authors gratefully acknowledge the Institute for the excellent working atmosphere provided. The work of the first author is part of the project 2010A2TFX2 {\it``Calcolo delle Variazioni''} funded by the Italian Ministry of Research and University. The first author is member of the Gruppo Nazionale per l'Analisi Matematica, la Probabilit\`a e le loro Applicazioni (GNAMPA) of the Istituto Nazionale di Alta Matematica (INdAM).


\bigskip
{\small\noindent
Giuseppe Buttazzo:
Dipartimento di Matematica,
Universit\`a di Pisa\\
Largo B. Pontecorvo 5,
56127 Pisa - ITALY\\
{\tt buttazzo@dm.unipi.it}\\
{\tt http://www.dm.unipi.it/pages/buttazzo/}

\bigskip\noindent
Edouard Oudet:
Laboratoire Jean Kuntzmann (LJK),
Universit\'e Joseph Fourier\\
Tour IRMA, BP 53, 51 rue des Math\'ematiques,
38041 Grenoble Cedex 9 - FRANCE\\
{\tt edouard.oudet@imag.fr}\\
{\tt http://www-ljk.imag.fr/membres/Edouard.Oudet/} 

\bigskip\noindent
Bozhidar Velichkov:
Laboratoire Jean Kuntzmann (LJK),
Universit\'e Joseph Fourier\\
Tour IRMA, BP 53, 51 rue des Math\'ematiques,
38041 Grenoble Cedex 9 - FRANCE\\
{\tt bozhidar.velichkov@imag.fr}\\
{\tt http://www.velichkov.it} 

\end{document}